\documentclass{article}[11pt]

\usepackage{amsmath}
\usepackage{amssymb}
\usepackage{caption}
\usepackage{amsfonts}
\usepackage{mathrsfs}
\usepackage{amsthm}
\usepackage{comment}
\usepackage{enumerate}
\usepackage{graphicx}
\usepackage[utf8]{inputenc}
\usepackage{bbm}
\usepackage{tikz}
\usepackage{float}
\usepackage[utf8]{inputenc}

\title{An explicit representation of the transition densities of the skew Brownian motion with drift and two semipermeable barriers}

\author{David Dereudre\footnote{Laboratoire de Math\'ematiques Paul Painlev\'e,  UMR CNRS 8524, Universit\'e Lille1, 59655 Villeneuve d’Ascq Cedex, France. 
david.dereudre@univ-lille1.fr}, 
Sara Mazzonetto\footnote{Institut f\"{u}r Mathematik der Universit\"{a}t Potsdam. Am Neuen Palais 10, 14469 Potsdam, Germany, and 
Laboratoire de Math\'ematiques Paul Painlev\'e, UMR CNRS 8524, Universit\'e Lille1, 59655 Villeneuve d’Ascq Cedex, France. \mbox{mazzonet@uni-potsdam.de}}, 
Sylvie Roelly\footnote{Institut f\"{u}r Mathematik der Universit\"{a}t Potsdam. Am Neuen  Palais 10, 14469 Potsdam, Germany, roelly@math.uni-potsdam.de}.}
\date{ }

\newtheorem{lemma}{Lemma}[section]

\newtheorem{remark}[lemma]{Remark}

\newtheorem{proposition}[lemma]{Proposition}

\newtheorem{theorem}[lemma]{Theorem}

\begin{document}

\maketitle
\begin{abstract}
In this paper, we obtain an explicit representation of the transition density of the one-dimensional skew Brownian motion with (a constant drift and) two semipermeable barriers. Moreover we propose a rejection method to simulate this density in an \emph{exact} way.
\end{abstract}



\section{Introduction}
The need to study the skew Brownian motion, and in particular its explicit transition densities, has emerged in various contexts during the last years. An overview and list of historical background and main applications can be found in \cite{AppShe}. 
Nevertheless, to the best of our knowledge, the transition densities of the one-dimensional skew Brownian motion with constant drift and two semipermeable barriers was not yet given as a closed formula, not even for the driftless version. In the latter case, one can only find in \cite{GOO} a \emph{non explicit formula} for it.

We obtain here a closed formula for the transition density of the skew Brownian motion with drift and two semipermeable barriers as series of Gaussian transition densities and cumulative distribution functions. This is a non trivial generalization of the case of reflecting barriers treated in \cite{Vee}.\\

In order to avoid repetitions, from now on we will use the following notation: $\beta$-SBM is the skew Brownian motion with one semipermeable barrier of permeability coefficient $\beta$ and $(\beta_1,\beta_2)$-SBM is the skew Brownian motion with two semipermeable barriers of permeability coefficients respectively $\beta_1$ and $\beta_2$.

The $\beta-$SBM was introduced by It\^o and McKean in \cite{IMcK}, as a one-dimensional Wiener process transformed by flipping the excursions from the origin with probability $\frac{1-\beta}{2} \in (0,1)$ (if $\beta=0$ it is the usual Brownian motion). Unfortunately this \emph{trajectorial definition} does not lend itself to generalizations.

The skew Brownian motion behaves as a Brownian motion between the barriers but it has a particular behaviour when it reaches them: it is partially reflected. This interpretation yields the various generalizations, that we are going to present shortly.\\

A recent survey on the skew Brownian motion can be found in \cite{Lej} in which various equivalent representations of the semigroup are given. Let us now present the process as a solution to a stochastic differential equation.

It was proved by Harrison and Shepp in \cite{HS} that  if $|\beta|\leq 1$, there is a unique strong solution to the {stochastic differential equation} involving the symmetric local time at the point 0 $(L^0_t)_{t\geq 0}$
\begin{equation} \label{skew}
\begin{cases}
dX_t=dW_t+ \beta \, d L^0_t(X), \\
X_0=x, \quad L^0_t=\int_0^t \mathbbm{1}_{\{X_s=0\}} d L^0_s,
\end{cases}
\end{equation}
that is the $\beta-$SBM. In particular if $\beta=0$, it is the usual Brownian motion. Harrison and Shepp also proved that if $|\beta|>1$ there is no solution to (\ref{skew}). Notice that if $x>0$, the $1-$SBM is the reflected Brownian motion on the positive semi-axis, and if $x<0$, the $(-1)-$SBM is the reflected Brownian motion on the negative semi-axis.\\

There are many possible generalizations of the SBM: one-dimensional skew BM with more semipermeable barriers (\cite{LeGall1},\cite{Ouk},\cite{Ram}), $n$-dimensional skew BM with one permeable barrier, as it is called in \cite{Ouk} referring to \cite{Por2} and \cite{Por}, and distorted Brownian motion (\cite{ORT}). A new proof of the weak existence and uniqueness for the $n-$dimensional SBM appeared recently in \cite{AB}.\\

The existence of several barriers does not allow anymore a trajectorial interpretation as randomly flipped excursions like for the $\beta$-SBM, nevertheless one can define the process as the unique strong solution to a slight modification of equation $(\ref{skew})$.
The stochastic differential equation $\mathcal{E}\left((\beta_1,\beta_2),\mu\right)$ satisfied by the $(\beta_1,\beta_2)-$SBM with drift $\mu\in\mathbb{R}$ is indeed
\begin{equation} \label{2skewmu}
\begin{cases}
dX_t= dW_t + \mu dt +  \beta_1 d  L^{z_1}_t(X)+ \beta_2 d  L^{z_2}_t(X),\\
X_0=x, \quad L^{z_1}_t=\int_0^t \mathbbm{1}_{\{X_s=z_1\}} d L^{z_1}_s, \  L^{z_2}_t=\int_0^t \mathbbm{1}_{\{X_s=z_2\}} d L^{z_2}_s,
\end{cases}
\tag*{$\mathcal{E}((\beta_1,\beta_2),\mu)$}
\end{equation}
where the coefficients $\beta_1,\beta_2 \in (-1,1)$ and $z_1,z_2\in \mathbb{R}$ are the barriers. 
Obviously, if $\beta_2=0$ the second barrier disappears and one obtains the equation satisfied by the $\beta_1$-SBM with drift with semipermeable barrier $z_1$:
\begin{equation} \label{1skewmu}
\begin{cases}
dX_t= dW_t + \mu dt +  \beta_1 d L^{z_1}_t(X),\\
X_0=x, \quad L^{z_1}_t=\int_0^t \mathbbm{1}_{\{X_s=z_1\}} d L^{z_1}_s.
\end{cases}
\tag*{$\mathcal{E}(\beta_1,\mu)$}
\end{equation}

The transition probability density function $p^{(\beta_1)}_\mu(t,x,y)$ of the Markov process, unique solution to \ref{1skewmu}, is computed in \cite{EMloc} using the trajectorial interpretation. As already noticed this approach is not extendable for finding the transition density in presence of more barriers. So let us briefly recall how to compute the semigroup of the $\beta-$SBM with barrier in zero as solution of a partial differential equation with specific boundary conditions.

In \cite{Por2} and \cite{Por} it is shown that
\begin{equation} \label{gen} 
\hat{L}=\frac{1}{2} \Delta +\beta \delta_0 \nabla
\end{equation}
is, formally, the infinitesimal generator of the $\beta-$ SBM with barrier in zero. Moreover the parabolic problem $\partial_t u=\hat{L} u$ (whose solution is the semigroup generated by $\hat{L}$) is equivalent to the transmission problem (see \cite{Lej}, section 3.1):
\begin{equation} \label{trasm}
\begin{cases}
\partial_t v =\frac{1}{2} \Delta v\\
(1+\beta) \nabla v(t,0^+)=(1-\beta) \nabla v(t,0^-) \qquad \textrm{(transmission condition).}
\end{cases}
\end{equation}
A solution to (\ref{trasm})
is equivalently a weak solution to the following problem:
\begin{equation} \label{div}
\begin{cases}
u(t,x)\in\mathcal{C}(0,T;L^2(\mathbb{R}))\cap L^2(0,T;H^1(\mathbb{R})), \\
\partial_t u =L u,\\
u(0,x)=\varphi(x)\in L^2(\mathbb{R}),
\end{cases}
\end{equation}
where $L$ is the divergence form operator 
\[L=\frac{1}{2 k(x)} \frac{d}{dx}\left(k(x)\frac{d}{dx}\right) \textrm{ with } k(x)= \frac12+\beta \left(\mathbbm{1}_{\mathbb{R}^+}(x)-\frac12 \right)
\]
with domain $\mathcal{D}(L)=\left\{ \varphi \in H^1(\mathbb{R}) ; \ k(x)\varphi '(x)\in H^1(\mathbb{R})\right\}$. Using Dirichlet forms one proves that the unique solution of (\ref{div}) is the semigroup of the $\beta-$SBM, solution of (\ref{skew}) (see section 3 of \cite{Lej}).

Our approach for computing the transition density of the $(\beta_1,\beta_2)-$SBM with or without drift will be based on identiying its infinitesimal generator as a divergence form operator, generalising the case of the driftless process treated for example in \cite{Etoreth},\cite{LM}.  
Once we will have computed the divergence form of the infinitesimal generator $(L,\mathcal{D}(L))$ associated to \ref{2skewmu}, we will solve the Kolmogorov equation satisfied by the semigroup: for each continuous and bounded function $f$, $P_t f$ is the solution in $ \mathcal{C}^{1,2}(\mathbb{R}_+\times \mathbb{R}\setminus{\{0\}},\mathbb{R})\cap \mathcal{C}(\mathbb{R}_+\times \mathbb{R},\mathbb{R})$ of
\begin{equation}\label{semig}
\begin{cases}
\frac{\partial}{\partial t}u(t,x)=L u(t,x) = \frac{1}{2}\frac{\partial^2}{\partial x^2}u(t,x) +\mu \frac{\partial}{\partial x} u(t,x)& t\in\mathbb{R}_+, \ x \in \mathbb{R}\setminus \{z_1,z_2\},\\
\frac{1+\beta_j}{2} \ \nabla u(t,z_j^+) =\frac{(1-\beta_j)}2 \ \nabla u(t,z_j^-)  & t\in \mathbb{R}_+, \  j=1,2, \\
u(t,z_j^+)=u(t,z_j^-) & t\in \mathbb{R}_+, \ j=1,2,\\
u(0,x)=f(x) & x\in\mathbb{R}.
\end{cases}
\end{equation}
The transition density $(t,y)\mapsto p(t,x,y)$ will satisfy, for $x$ fixed, the analogous PDE for the adjoint $L^*$:
\begin{equation}\label{traden}
\begin{cases}
\frac{\partial}{\partial t}p(t,y)=L^* p(t,y) & t\in(0,+\infty), \ y \in \mathbb{R}\setminus \{z_1,z_2\},\\
\frac{1}{2} \nabla p(t,z_j^+) \, - \, \mu \, p(t,z_j^+) = \frac{1}{2} \nabla p(t,z_j^-) \, - \, \mu \, p(t,z_j^-) & t\in \mathbb{R}_+, \  j=1,2, \\
(1+\beta_j) \,  p(t,z_j^-)=(1-\beta_j) \, p(t,z_j^+) & t\in \mathbb{R}_+, \ j=1,2,\\
p(0,y)=\delta_{x}(y) & y\in\mathbb{R}.
\end{cases}
\end{equation}

The paper is organized as follows: in Section~\ref{reproper} we give an explicit characterization of the infinitesimal generator associated to the solution to \ref{2skewmu} in order to obtain a representation of its transition density. Then we exploit it in the following cases: first for the $(\beta_1,\beta_2)-$SBM without drift, then for the $\beta-$SBM with constant drift, and finally we give the formula for the drifted version of the $(\beta_1,\beta_2)-$SBM.
Moreover we discuss some particular and limit cases and compare our results to former ones. In Section~\ref{exactsimulation} we present a rejection sampling method that allows to simulate exactly the SBM with two semipermeable barriers.

\section{The transition density of the $(\beta_1,\beta_2)$-SBM with and without drift} \label{reproper}
\subsection{The framework and the method}\label{method}

In order to obtain the transition density of the $(\beta_1,\beta_2)$-SBM, we identify its infinitesimal generator. The infinitesimal generator of the $\beta_1$-SBM with one semipermeable barrier in $z_1$, solution of the equation $\mathcal{E}(\beta_1,0)$, is the divergence form operator
\begin{equation} \label{Lkoperator}
\begin{cases}L=\frac{1}{2 k(x)}\frac{d}{dx}\left(k(x)\frac{d}{dx}\right), \quad \mathcal{D}(L)=\left\{ \psi \in H^1(\mathbb{R}) ; \ k(x)\psi'(x)\in H^1(\mathbb{R}) \right\},\\
k(x)= \frac12 + \beta_1\left(\mathbbm{1}_{[z_1,+\infty)}(x)- \frac12\right)
\end{cases}
\end{equation}
with piecewise constant function $k(x)$ unique up to a multiplicative constant. 
(See for example \cite{LM}).

Notice that a straightforward generalization of (\ref{Lkoperator}) yields to the generator of the $\beta=(\beta_1,\beta_2,\ldots,\beta_n)$-SBM with $n$ semipermeable barriers in $z_1<z_2<\ldots<z_n$, by modifying the piecewise constant function $k(x)$. Indeed, in case of two semipermeable barriers $z_1,z_2$, the function $k(x)$ assumes three different values:
\begin{equation} \label{k2skew}
k(x)=\left(\frac12 +\beta_1\left(\mathbbm{1}_{[z_1,+\infty)} (x) -\frac12 \right) \right)\left(\frac12 +\beta_2\left(\mathbbm{1}_{[z_2,+\infty)}(x)-\frac12 \right)\right)=
\begin{cases} \frac14 (1-\beta_1)(1-\beta_2) & x<z_1 ,\\
\frac14 (1+\beta_1)(1-\beta_2) & z_1 \leq x < z_2,\\
\frac14(1+\beta_1)(1+\beta_2) & x\geq z_2.
\end{cases}
\end{equation}
Therefore taking $k=\prod_{m=0}^n \left(\frac12 +\beta_m\left(\mathbbm{1}_{[z_m,+\infty)}-\frac12 \right) \right)$, $(L,\mathcal{D}(L))$ is the infinitesimal generator of the SBM with $n$ semipermeable barriers.\\

The operator $L$ is a divergence form operator with discontinuous coefficients, therefore one obtains a representation for the transition densities, as in \cite{GOO}, Chapter II. In section 5 the authors recover themselves the case of the $\beta$-SBM and the $(\beta_1,\beta_2)$-SBM without drift. Unfortunately in the latter case the authors do not explicit further the transition density function, they just identify it through a ``kind of $\theta$-function''
\begin{equation}  \label{thetafunction}
h(t,\xi,C,\alpha)=\frac{1}{2\pi}\int_{-\infty}^{+\infty}e^{- w^2 t}e^{i w\xi}\left(1+Ce^{i w \alpha}\right)^{-1} dw.
\end{equation}
We are therefore going to generalize the method, called Green function method or Titchmarsh-Kodaira-Yoshida method, for giving an \textbf{explicit} representation of the transition density function associated to the slightly more general infinitesimal generator including a constant drift $\mu\in \mathbb{R}$.
\begin{equation} \label{operator}
\begin{cases} L=\frac{1}{2 h(x)}\frac{d}{dx}\left(h(x)\frac{d}{dx}\right), \\
\mathcal{D}(L)=\left\{ \psi \in H_0^1(h(x)dx) ; \ h(x)\psi'(x)\in H^1(h^{-1}(x)dx) \right\},\\
h(x):=k(x) e^{2 \mu x} \end{cases}
\end{equation}
where $k(x)$ is the piecewise constant function defined in (\ref{k2skew}). Remark that $h(x)$ is strictly positive but not bounded from above.

\begin{lemma}\footnote{The authors would like to thank Markus Klein (Universit\"{a}t Potsdam) for the interesting discussions}
The operator $(L,\mathcal{D}(L))$ defined by (\ref{operator})
\begin{enumerate}[(i)]
\item is self-adjoint in $L^2(h(x)dx)$ and its spectrum $\sigma(L)$  is a closed subset of $(-\infty,0]$ containing $0$;
\item is the infinitesimal generator in $L^2(h(x)dx)$ of the process $(\beta_1,\beta_2)$-SBM with drift $\mu$.
\end{enumerate}
\end{lemma}
\begin{proof}
First of all notice that the measure $\nu(dx):=h(x)dx$ is not a finite measure.
The form 
\begin{equation} \label{formq}
q: H^1_0(h(x)dx) \times H^1_0(h(x)dx) \to \mathbb{R} \text{  defined by  } (f,g)\mapsto \int_\mathbb{R} f' g'  h(x)dx,
\end{equation}
is symmetric, semibounded and closed with domain $\mathcal{Q}(q)=H^1_0(h(x)dx) \subseteq L^2(h(x)dx)$. 
Therefore there exists a unique operator $T$ with $\mathcal{D}(T)\subseteq \mathcal{Q}(q)$ such that $q(u,v)=-<u,Tv>_\nu$ (see for example Corollary 1.3.1 in \cite{Fuku}). Moreover the operator $T$ is self-adjoint.\\
One can show that $(2 L,\mathcal{D}(L)) =(T,\mathcal{D}(T))$ using the implicit characterization of $\mathcal{D}(T)$. Therefore the operator $2 L$ is self-adjoint, hence the conclusions on its spectrum.

We now apply the results presented in the recent paper \cite{ORT}, Remark 2.6.ii: the Hunt process whose semigroup is associated to the closed form $\left(q,\mathcal{Q}(q)\right)$ in (\ref{formq}) is the SBM with drift with semipermeable barriers. By uniqueness of the self-adjoint operator associated to the form we conclude that the operator $\left(L,\mathcal{D}(L)\right)$ is the infinitesimal generator of this SBM.
\end{proof}

\begin{remark}
\begin{enumerate}[(i)]
\item The same lemma holds for the $\beta_1$-SBM with drift, 
and also for the driftless processes $(\beta_1,\beta_2)$-SBM and $\beta_1$-SBM ($\mu=0$); 
\item As an alternative to (\ref{operator}), one can express the infinitesimal generator for the $(\beta_1,\beta_2)$-SBM with drift as
\[\begin{cases} A=\frac{1}{2 k(x)}\frac{d}{dx}\left(k(x) \frac{d}{dx}\right)+\mu \frac{d}{dx}\\
\mathcal{D}(A):=\left\{ \psi \in H^1(dx); \ k \psi ' \in H^1(dx) \right\}
\end{cases}\]
with $k(x)$ defined in (\ref{k2skew}). In that case, $A$ is not self-adjoint.
\item One can show, using Hille-Yoshida theorem, that the operator $(L,\mathcal{D}(L))$ is sectorial since it is self adjoint and in particular it is the infinitesimal generator af a strongly-continuous semigroup of contractions.
\end{enumerate}
\end{remark}

Since the infinitesimal generator $(L,\mathcal{D}(L))$ is a sectorial operator, its associated transition semigroup $P_t$ can be represented as:
\[ P_t\varphi(x) =\frac{1}{2\pi i}\int_\Gamma e^{\lambda t} u_{\lambda,\varphi} (x) d\lambda\]
where $\Gamma$ is a contour in the complex $\lambda$ plane around the negative semi-axis $(-\infty,0]$ that contains the spectrum $\sigma(L)$ and $u_{\lambda,\varphi}$ is the resolvent solution to $(\lambda-L)u_{\lambda,\varphi}=\varphi$ for all $\varphi\in L^2(h(x)dx)$ (see for example Theorem 12.31 in \cite{RR}). 
Therefore the transition density satisfies 
\begin{equation} \label{representationdensity}
p (t,x,y)=\frac{1}{2\pi i}\int_{\Gamma} e^{\lambda t} G(x,y;\lambda) d\lambda
\end{equation}
where $G(x,y;\lambda)$ are the Green functions.

\begin{lemma} \label{Greenfunctions}
For each $\lambda \in \mathbb{C}\setminus \mathbb{R}_-$ the Green functions are given by 
\[G(x,y;\lambda) = -2h(y)\frac{ U_+(y,\lambda) U_-(x,\lambda)\mathbbm{1}_{\{x\leq y\}}+U_+(x,\lambda) U_-(y,\lambda)\mathbbm{1}_{\{y<x\}}}{h(x_0)W(U_-,U_+)(x_0,\lambda)}\]
where $U_\pm \in \mathcal{D}(L)$ are the solutions to:
\begin{equation}\label{autoval}
(\lambda -L) U_\pm(x,\lambda)=0, \qquad \lim_{x\to + \infty}\,U_+ (x,\lambda) =0, \quad \lim_{x\to - \infty}\,U_- (x,\lambda) =0, 
\end{equation}
while $W(U_-,U_+)(x_0,\lambda)=U_-(x_0,\lambda)U_+ ' (x_0,\lambda)-U_-'(x_0,\lambda)U_+(x_0,\lambda)$ is the wronskian in $x_0\in\mathbb{R}$.
\end{lemma}
\begin{proof}
One can easily prove that the function $x\mapsto h(x)W(U_-,U_+)(x,\lambda)$ is constant and check that $x\mapsto G(x,y;\lambda)$ is a solution to $(\lambda-L)v(x)=\delta_{0}(x-y)$ for all $y\in \mathbb{R}$. By uniqueness of the solution the proof is done.
\end{proof}
\begin{remark} Notice that $x \mapsto G(x,y;\lambda) \in \mathcal{D}(L)$, and $y\mapsto G(x,y;\lambda)\in \mathcal{D}(L^*)=\left\{\varphi, \frac{\varphi}{g}\in \mathcal{D}(L)\right\}$ since $L^*g:=hL\left(\frac{g}{h}\right)$ is the adjoint in $L^{2}(dx)$. The same holds for $y\mapsto p(t,x,y)\in\mathcal{D}(L^*)$.
\end{remark}

\subsection{The case of $(\beta_1,\beta_2)$-SBM without drift}
We will now present the method step by step.

\subsubsection{The Green functions} \label{sectionGreen2skew}

The first step is to find the eigenfunctions $U_{+}(x,\lambda)$ and $U_-(x,\lambda)$ of $L$ defined in (\ref{autoval}). The two barriers divide the real line into three intervals over which the functions $U_\pm$ can be constructed as linear combination of the eigenfunctions of the operator $L$ for the eigenvalue $\lambda\in \mathbb{C}\setminus (-\infty, 0]$ that are $u_-(x)=\exp{\left(\sqrt{2\lambda}x \right)},\quad u_+(x)=\exp{\left(-\sqrt{2\lambda}x \right)}.$ Therefore

\[U_-=\begin{cases}
u_- & x\leq z_1,\\
A(\lambda) u_- + B(\lambda) u_+ & z_1\leq x\leq z_2, \\
C(\lambda) u_- + D(\lambda) u_+ & x\geq z_2;
\end{cases}\]
and
\[U_+=\begin{cases}
G(\lambda) u_- + H(\lambda) u_+ & x \leq z_1, \\
E(\lambda) u_- + F(\lambda) u_+ & z_1\leq x \leq z_2, \\
u_+& x \geq z_2,
\end{cases}\]
with eight coefficients to be determined. Notice that since $U_\pm\in \mathcal{D}(L)$, they are continuous functions and have to satisfy the so-called \emph{transmission conditions} derived from the continuity of $x\mapsto k(x)U_\pm(x,\lambda)$. These conditions will determine uniquely the eight coefficients:
\[ \begin{cases}
A(\lambda) = {  (1+\beta_1)^{-1}};\\
B(\lambda) = A(\lambda) {\beta_1 e^{  2 \sqrt{2\lambda} z_1}}\\
C(\lambda) = \left({\beta_1 \beta_2 e^{-2 \sqrt{2\lambda}(z_2-z_1)} + 1}\right){\left( (1+\beta_1 )(1+\beta_2)\right)^{-1}};\\
D(\lambda) = {\left(\beta_1 e^{2 \sqrt{2\lambda} z_1} + \beta_2 e^{  2 \sqrt{2\lambda} z_2 }\right)}{\left( (1+\beta_1 )(1+\beta_2)\right)^{-1}}
\end{cases}
\]
\[\begin{cases}
G(\lambda)=
-\left(\beta_2 e^{ - 2 \sqrt{2\lambda} z_2} + \beta_1 e^{ - 2 \sqrt{2\lambda} z_1} \right){\left(  (1 - \beta_1) (1 - \beta_2) \right)^{-1}}\\
H(\lambda)=
\left({\beta_1\beta_2 e^ {-2 \sqrt{2\lambda} (z_2-z_1)} + 1}\right)
{\left((1 - \beta_1) (1 - \beta_2) \right)^{-1}},\\
E(\lambda)=-
F(\lambda){\beta_2 e^{-2 \sqrt{2\lambda} z_2}}, \\
F(\lambda)={(1-\beta_2)^{-1}}.
\end{cases}\]

The second step is to compute \textbf{the wronskian}. Consider $x_0<z_1$, for example $x_0=z_1-1$. Hence the wronskian is
\[\begin{split} W(x_0) & = -\sqrt{2 \lambda}\frac{\beta_1 \beta_2 e^{-2 \sqrt{2 \lambda} z} +1}{2  k(x_0)} =-2 \sqrt{2\lambda} H(\lambda),\end{split}\]
where we will denote by $z$ the distance between the barriers $z_2-z_1$. This leads to the following
\begin{proposition}\label{greenunique2skew}

The Green functions are given by
\[G(x,y;\lambda) = \frac{1}{\phi(\lambda)}  e^{- \phi(\lambda) |x-y|} \frac{ \sum_{j=1}^4 c_j(y,\beta_1,\beta_2)e^{-\phi(\lambda) a_j(x,y)}}{\beta_1 \beta_2 e^{-2 \phi(\lambda) z}  +   1}.\]
where $\phi(\lambda):=\sqrt{2\lambda}$, $z:=z_2-z_1$ is the distance between the barriers, and

\[\begin{cases}
c_1(y,\beta_1,\beta_2)\equiv 1 		\\
c_2(y,\beta_1,\beta_2)=\left(2\mathbbm{1}_{[z_1,+\infty)}(y)-1\right)\beta_1 \\			
c_3(y,\beta_1,\beta_2)=\left(2\mathbbm{1}_{[z_2,+\infty)}(y)-1\right)\beta_2 \\			
c_4(y,\beta_1,\beta_2)=\left(1-2\mathbbm{1}_{[z_1,z_2)}(y)\right)\beta_1\beta_2 			
\end{cases}
\begin{cases}
a_1(x,y)\equiv 0\\
a_2(x,y)=|y-z_1|+|x-z_1|-|y-x|\\
a_3(x,y)=|y-z_2|+|y-z_2|-|y-x|\\
a_4(x,y)=2\left(z_2-max(x,y,z_1)\right)^+ + 2 \left(min(x,y,z_2)-z_1\right)^+
\end{cases}\]
\end{proposition}

\begin{proof} We will only do the computations in the case $x<z_1<z_2<y$, the other cases are similar. From Lemma \ref{Greenfunctions}, since $h \equiv k$ and chosing $x_0<z_1$, the Green function is of the following form
\[G(x,y;\lambda) = -2k(y)\frac{ U_+(y,\lambda) U_-(x,\lambda)}{k(x_0)W(U_-,U_+)(x_0)}=-2k(y)\frac{ u_+(y,\lambda) u_-(x,\lambda)}{- \frac12 \sqrt{2\lambda}(1+\beta_1\beta_2 e^{-2\sqrt{2\lambda} z})}= \frac1{\sqrt{2\lambda}} \frac{4 k(y) e^{-\sqrt{ 2\lambda} (y-x)}}{1+\beta_1\beta_2 e^{-2\sqrt{2\lambda} z}}.\]
It is then sufficient to check that $a_j(x,y)= 0$ for $j\in\{1,2,3,4\}$ and $\sum_{j=1}^4 c_j(y,\beta)=4 k(y)$.
\end{proof}

\begin{remark}
\begin{enumerate}[(i)]
\item The function $\phi$ is well defined as bijection between $\mathbb{C}\setminus {(-\infty,0]}$ and $\{ \zeta \in \mathbb{C}; \ \Re(\zeta)>0\}$.
\item The denominator $\lambda \mapsto 1+\beta_1\beta_2 e^{-2 \phi(\lambda) z}$ has no zero in $\mathbb{C}\setminus {(-\infty,0]}$ since $\Re \phi(\lambda)>0$.
\item $a_j(x,y)\geq 0$ for $j\in\{1,2,3,4\}$.
\end{enumerate}
\end{remark}

\subsubsection{The transition density as (contour) integral} \label{changegreen}

Since the Green functions depend on $\lambda$ only through $\phi(\lambda)=\sqrt{2\lambda}$, we can apply the change of variable $\displaystyle{\lambda \mapsto \phi(\lambda)=:\xi}$ to the integral appearing in (\ref{representationdensity}):
\[\int_\Gamma e^{\frac{\phi(\lambda)^2}{2} t} \ \overline{G}(x,y;\phi(\lambda)) \ d\phi(\lambda)=\int_{\phi(\Gamma)} e^{\frac{\xi^2}{2} t} \ \overline{G}(x,y;\xi) \ d\xi\]
where $\overline{G}(x,y;\phi(\lambda))= \phi(\lambda) G(x,y;\lambda)$ 
and $G(x,y,\lambda)$ given in Proposition \ref{greenunique2skew}.\\
Since the integrand  $e^{\frac{\xi^2}2 t}\overline{G}(x,y;\xi)$ is holomorphic in the closed subset of the complex plane between $i \mathbb{R}$ and $\phi(\Gamma)$, we could deform (shrink) the contour $\phi(\Gamma)$ to the imaginary line by an homotopy. Indeed, if we denote by $M$ the unique point with imaginary part $u$ in $\phi(\Gamma)$ (as in Figure~\ref{contourshrink}), it is possible to shrink the contour $\phi(\Gamma)$ to $i \mathbb{R}$ if the following lemma holds:

\begin{lemma} \label{smallintegral}
Consider the function 
\[\overline{G}(x,y,\xi)= e^{- \xi |x-y|} \frac{ \sum_{j=1}^4 c_j(y,\beta_1,\beta_2)e^{-\xi a_j(x,y)}}{\beta_1 \beta_2 e^{-2 \xi z}  +   1}.\] 
Then
\[\lim_{u \to \pm\infty} \int_{\rho_M} e^{\frac{\xi^2}2 t} \ \overline{G}(x,y;\xi) \ d\xi =0, \]
where $\rho_M$ is the segment in the figure connecting the point  $M$ with its projection on $i \mathbb{R}$, $M'=(0,u)$.
\end{lemma}

\begin{proof}
Let us show that the absolute value converges to zero:
\[\left| \int_{\rho_M} e^{\frac{\xi^2}{2} t} \ \overline{G}(x,y;\xi) \ d\xi \right| \leq \int_{\rho_M} \left| e^{\frac{\xi^2}{2} t} \ \overline{G}(x,y;\xi)\right| \ d\xi =  \int_0^{\ell(u)} e^{\frac{(v^2 -u^2)}{2} t} \ \left|\overline{G}(x,y;i u+ v)\right| \ dv \]
with $\ell(u) := \left|M'-M \right|$ (hence $M=(\ell(u),u)$) and $\lim_{u\to\infty} \ell(u)=0$.
Let us notice that
\[\left|\overline{G}(x,y;i u+ v)\right|\leq e^{-v |x-y|} \frac{ \sum_{j=1}^4 \left| c_j(y,\beta_1,\beta_2)\right| e^{- v a_j(x,y)}}{\left|\beta_1 \beta_2 e^{-2 (i u +v) z}  +   1\right|}\leq e^{-v |x-y|} \frac{ \sum_{j=1}^4 \left| c_j(y,\beta_1,\beta_2)\right| e^{- v a_j(x,y)}}{1- \left|\beta_1 \beta_2\right| e^{-2 v z}},\]
therefore
\[\left| \int_{\rho_M} e^{\frac{\xi^2}{2} t} \ \overline{G}(x,y;\xi) \ d\xi \right| \leq e^{-\frac{u^2}2 t}\int_0^{\ell(u)} e^{\frac{v^2}{2} t} \ e^{-v |x-y|} \frac{ \sum_{j=1}^4 \left| c_j(y,\beta_1,\beta_2)\right| e^{- v a_j(x,y)}}{1- \left|\beta_1 \beta_2\right| e^{-2 v z}} \ dv\]
that clearly converges to zero if $\left|u\right|$ goes to infinity.
\end{proof}

\begin{figure}[H]
\begin{center}
\begin{tikzpicture}
\path(-2.8,2.8) node{a)};
\draw[-stealth] (-3,0) -- (3,0) node[above left]{$\mathbb{R}$};
\draw[-stealth] (0,-3) -- (0,3) node[below right]{$i\mathbb{R}$};

\draw[red] (-3,0) -- (0,0);
\fill[red] (0,0) circle(0.05);

\draw[blue] (-3,-1) -- node[above]{$\Gamma$} (0,-1) arc(-90:90:1)  -- (-3,1);
\draw[blue,->] (-2,-1) -- (-1,-1);
\draw[blue,->] (-1.5,1) -- (-2,1);

\draw[green] (0.5,-3) ..controls(0.6,-2.5) and (0.7,-2.1).. (1.4,-1.4) arc(-45:45:2) ..controls(0.7,2.1) and (0.6,2.5).. node[above right]{$\phi(\Gamma)$} (0.5,3);
\draw[green,->] (2,0) arc(0:5:2);

\end{tikzpicture} 
\hspace{3cm}
\begin{tikzpicture}
\path(-2.3,2.8) node{b)};
\draw[-stealth] (-1,0) -- (3,0) node[above left]{$\mathbb{R}$};
\draw[-stealth] (0,-3) -- (0,3);

\draw[green] (0.5,-3) ..controls(0.6,-2.5) and (0.7,-2.1).. node[below right]{$\phi(\Gamma)$} (1.4,-1.4) arc(-45:45:2) ..controls(0.7,2.1) and (0.6,2.5).. (0.5,3);

\draw[magenta] (0,2.25) node[left]{$M'=(0,u)$} -- node[below]{$\rho_M$} (0.73,2.25) node[right]{$M$};
\fill[magenta] (0,2.25) circle (0.05) (0.73,2.25) circle (0.05);
\draw[magenta] (0,-2.25) node[left]{$-M'=(0,-u)$} -- (0.73,-2.25);
\fill[magenta] (0,-2.25) circle (0.05);

\end{tikzpicture}
\captionsetup{singlelinecheck=off}
\caption[countour]{
\begin{enumerate}[a)]
\item The picture shows the green image of the blue contour $\Gamma$ under $\phi$. The spectrum of the operator $(L,\mathcal{D}(L))$ is contained in the red semi-axis $(-\infty,0]$, which coincides with the complement of the domain of $\phi$.
\item In this figure one sees the magenta segment $\rho_M$ connecting the unique point $M$ in $\phi(\Gamma)$ with imaginary part $u$ to its projection $M'$ on the imaginary line. The homotopy $H:[0,1]\times \mathbb{R}\to \mathbb{R}^2$ that deforms $\phi(\Gamma)$ into $i \mathbb{R}$ is given by $H(t,u)=M'(1-t)+t M$.
\end{enumerate}
}\label{contourshrink}
\end{center}
\end{figure}

Therefore the integral in (\ref{representationdensity}) becomes (with $\xi=i w$)
\begin{equation} \label{tdf2skewwip}
p^{(\beta_1,\beta_2)}(t,x,y)=\frac{1}{2\pi} \int_{\mathbb{R}} e^{-\frac{w^2}{2}t}  e^{- i w |x-y|} \frac{ \sum_{j=1}^4 c_j(y,\beta_1,\beta_2)e^{-i w a_j(x,y)}}{\beta_1 \beta_2 e^{-2 i w z}  +   1} dw.
\end{equation}
One can also rewrite it using the functions $h$ defined by equation (\ref{thetafunction}):
\[p^{(\beta_1,\beta_2)}(t,x,y)= \sum_{j=1}^4 c_j(y,\beta_1,\beta_2) h\left(\frac t 2, -\left(a_j(x,y)+|x-y|\right),\beta_1\beta_2,-2z\right),\]
which agrees with the results in \cite{GOO}. Nevertheless, since $\left|\beta_1\beta_2\right|<1$, we can explicit further the expression (\ref{tdf2skewwip}).

\subsubsection{The transition density as series of Fourier transforms}
\begin{proposition} \label{transitiondensity2skew}
The transition density of the $(\beta_1,\beta_2)$-SBM has the following expansion 
\begin{equation} \label{tdf2skew}
p^{(\beta_1,\beta_2)}(t,x,y)=p^{(0,0)}(t,x,y) \sum_{k=0}^{\infty}  (-\beta_1 \beta_2)^k  \sum_{j=1}^4 c_j(y,\beta_1,\beta_2)e^{ -\frac{\left(a_j(x,y)+2 z k\right)^2}{2 t}}  e^{-|x-y|\frac{a_j(x,y)+2 z k}{t}}
\end{equation}
\end{proposition}
where $p^{(0,0)}(t,x,y)$ is the transition density function of the Brownian motion.

\begin{proof}
Let us consider the expression (\ref{tdf2skewwip}). The denominator can be seen as the sum of a geometric series
\[\frac{1}{1+\beta_1\beta_2 e^{-2 i w z}}=\sum_{k=0}^\infty (-\beta_1\beta_2)^k e^{-2 i w z k}\]
since $|\beta_1\beta_2 e^{-2 i w z}|=|\beta_1\beta_2|<1$.

Therefore the density can be written as
\[\begin{split}
p^{(\beta_1,\beta_2)}(t,x,y) & =\frac{1}{2\pi }\int_{\mathbb{R}}  \sum_{k=0}^{\infty}  (-\beta_1 \beta_2)^k e^{-\frac{w^2}{2}t} e^{-i w |x-y|} \sum_{j=1}^4 c_j(y,\beta_1,\beta_2) e^{- iw (a_j(x,y)+2 z k)} dw.
\end{split}\]

We can exchange integral and series, because the series of absolute values $e^{-\frac{w^2}{2}t} \frac{1}{1-|\beta_1 \beta_2|}$ is integrable.

We conclude that the transition density is a series of Fourier transforms:

\begin{equation}\label{Fourier2skew}
\begin{split}
p^{(\beta_1,\beta_2)}(t,x,y) & =\sum_{k=0}^{\infty}  (-\beta_1 \beta_2)^k  \sum_{j=1}^4 c_j(y,\beta_1,\beta_2) \frac{1}{2\pi }\int_{\mathbb{R}}  e^{-\frac{w^2}{2}t} e^{-i w |x-y|}  e^{- iw (a_j(x,y)+2 z k)} dw\\
& = \frac{1}{\sqrt{2\pi }} \sum_{k=0}^{\infty}  (-\beta_1 \beta_2)^k  \sum_{j=1}^4 c_j(y,\beta_1,\beta_2)  \hat g_t (a_j(x,y)+2 z k+|x-y|)\\
& =  \frac{1}{\sqrt{2\pi t}} \sum_{k=0}^{\infty}  (-\beta_1 \beta_2)^k  \sum_{j=1}^4 c_j(y,\beta_1,\beta_2)  g_1 \left(\frac{a_j(x,y)+2 z k+|x-y|}{\sqrt{t}}\right)
\end{split}
\end{equation}
where $g_t(w):=e^{-\frac{w^2}{2} t}=g_1(w \sqrt{t})$ and its Fourier transform satisfies $\hat g_t(\omega)=\frac{1}{\sqrt{t}} g_t(\frac {\omega}{ t })= \frac{1}{\sqrt{t}} g_1\left(\frac{\omega}{\sqrt{t}}\right). $\\
We notice that $g_1(a+b)=g_1(a)g_1(b)e^{-a b}$ hence we can write the density as
\[
p^{(\beta_1,\beta_2)}(t,x,y) = \frac{1}{\sqrt{2 \pi t}} \, g_1\left(\frac{|x-y|}{\sqrt{t}}\right)\sum_{k=0}^{\infty}  (-\beta_1 \beta_2)^k  \sum_{j=1}^4 c_j(y,\beta_1,\beta_2) g_1 \left(\frac{a_j(x,y)+2 z k}{\sqrt{t}}\right) e^{-\frac{|x-y|}{t}\left(a_j(x,y)+2 z k\right)}.
\]
Using the identity $p^{(0,0)}(t,x,y)=\frac{1}{\sqrt{2 \pi t}} g_1\left(\frac{|x-y|}{\sqrt{t}}\right)$ we conclude and obtain (\ref{tdf2skew}).
\end{proof}

\subsection{The case of the $(\beta_1,\beta_2)$-SBM with drift}

\subsubsection{Expansion of the transition density in the case of one barrier and drift}
In this subsection we propose an explicit computation of the transition density function of the $\beta_1$-SBM with constant drift $\mu$, solution to \ref{1skewmu}.

\begin{proposition} \label{tdf1skewmu}
The transition density for the SBM with constant drift $\mu$ and barrier in $z_1$ satisfies
\[p^{(\beta_1)}_{\mu}(t,x,y)=p^{(0)}_{\mu}(t,x,y)v^{(\beta_1)}_\mu(t,x,y),\]
where $p^{(0)}_\mu(t,x,y)$ 
is the transition density of the Brownian motion with drift $\mu$ (without skew), and 
\[\begin{split} v^{(\beta_1)}_\mu(t,x,y):= & \left(1 - exp{\left(-\frac{2x_1y_1}{t}\right)}\right) \mathbbm{1}_{\{x_1y_1 >0\}}+ \left[1+\beta_1 \left(2 \mathbbm{1}_{[z_1,+\infty)}(y) - 1\right)\right] \exp{\left(-\frac{2x_1y_1}{t} \mathbbm{1}_{\{x_1y_1 >0\}}\right)} \cdot \\
 & \cdot 
\left[1-\beta_1\mu\sqrt{2\pi t}\exp{\left(\frac{\left(|x_1|+|y_1|+t\beta_1 \mu\right)^2}{2t}\right)}\Phi^c\left(\frac{|x_1|+|y_1|+t\beta_1 \mu}{\sqrt{t}}\right)\right],
\end{split}\] 
where $x_1:=x-z_1$, $y_1:=y-z_1$ and $\Phi^c(y):=\frac{1}{\sqrt{2\pi}}\int_y^\infty e^{-\frac{u^2}{2}} du$ is the queue of a standard Gaussian law.
\end{proposition}
This result appears in \cite{EMloc} for a barrier in zero although it holds for any barrier. The authors prove it using the trajectorial definition of the SBM.\\
We provide here a completely different proof based on the generalization of the Green function method. Indeed the infinitesimal generator of the process is a generalization of (\ref{Lkoperator}) with the function $h(x)=k(x)e^{2\mu x}$ instead of $k(x)$. 
We will denote by $\beta:=\beta_1 \in (-1,1)\setminus \{0\}$ the unique skewness parameter.
The same method we develop here, will also provide the transition density for the $(\beta_1,\beta_2)$-SBM with drift even though trickier technical issues are involved.

\subsubsection*{A. The Green functions}
When there is a drift $\mu\neq 0$, the functions $U_\pm(x,\lambda)$ solutions to (\ref{autoval}) are linear combinations of the two eigenfunctions for the eigenvalue $\lambda\in \mathbb{C}\setminus (-\infty,0]$ $u_{\pm}(x,\lambda)=\exp{\left(-\mu x \mp \sqrt{\mu^2+2\lambda}\, x\right)}$. The coefficients are uniquely determined using the continuity and the transmission conditions:
\[U_-=\begin{cases}
u_- & \text{ on } (-\infty,z_1],\\
A(\lambda) u_- + B(\lambda)  u_+ & \text{ on } [z_1,+\infty),
\end{cases} \quad 
\textrm{and} 
\quad U_+=\begin{cases}
E(\lambda) u_- + F(\lambda) u_+ &  \text{ on } (-\infty,z_1], \\
u_+ & \text{ on } [z_1,+\infty),
\end{cases}\]
with
\[\begin{cases}  A(\lambda) = \frac{1}{\beta+1} \left(1+\frac{\beta\mu}{\sqrt{\mu^2+2\lambda}}\right), & B(\lambda) =(1-A(\lambda))e^{2\sqrt{\mu^2+2\lambda}z_1}\\
F(\lambda) =\frac{1}{(1-\beta)} \left(1+\frac{\beta\mu}{\sqrt{\mu^2+2\lambda}}\right), & E(\lambda)=(1-F(\lambda))e^{-2\sqrt{\mu^2+2\lambda}z_1}.\end{cases}
\]
We compute the wronskian at the point $x_0<z_1$ and obtain 
\[W(U_-,U_+)(x_0,\lambda)= -2 \sqrt{\mu^2+2\lambda} F(\lambda) \exp{\left(-2\mu x_0\right)}.\]
This leads to the result:
\begin{lemma}
The Green functions satisfy
\begin{equation}\label{g1skew}
\begin{split}
G(x,y;\lambda) & =2 h(y) \frac{U_-(x\wedge y,\lambda) U_+(x \vee y,\lambda)}{\beta \mu + \sqrt{2\lambda + \mu^2}}\\
& = e^{\mu (y-x)}  \frac{1}{\sqrt{2\lambda+\mu^2}} \ \frac{e^{-\sqrt{2\lambda+\mu^2}|y-x|} }{\sqrt{2\lambda +\mu^2}+\beta\mu}\left(\sum_{j=1}^2 c_j(\mu,y;\sqrt{2\lambda+\mu^2}) e^{-\sqrt{2\lambda +\mu^2}a_j(x,y)}\right)
\end{split}
\end{equation}
where
\[\begin{cases}
c_1(\mu,y; w) = {\beta \mu + w} \\
c_2(\mu,y; w)=  \beta w \left(2 \mathbbm{1}_{[z_1,+\infty)}(y)-1\right)-\beta \mu,
\end{cases} \begin{cases}
a_1(x,y)\equiv 0 \\
a_2(x,y)=|y-z_1|+|x-z_1|-|y-x|.
\end{cases}\]
\end{lemma}
Notice that $a_2(x,y)\geq 0$ for all $x,y \in \mathbb{R}$.

\subsubsection*{B. The transition density as a contour integral}
The dependence on $\lambda$ of the Green functions given by (\ref{g1skew}) is actually a dependence on
\[\phi(\lambda) :=\sqrt{2\lambda + \mu^2}\in \{ \zeta \in \mathbb{C}; \Re(\zeta)>0\}.\]
This allows the change of variables $\xi:=\phi(\lambda)$ as in the subsection \ref{changegreen}:

\[\begin{split}&  p^{(\beta)}_\mu(t,x,y)=\frac{1}{2\pi i}\int_\Gamma e^{\lambda t}G(x,y;\lambda)d\lambda=\\
& =  e^{\mu (y-x)} \frac{1}{2\pi i}\int_\Gamma e^{\lambda t} \frac{e^{-\sqrt{2\lambda+\mu^2}|y-x|} }{\sqrt{2\lambda +\mu^2}+\beta\mu} \left(\sum_{j=1}^2 c_j(\mu,y;\sqrt{2\lambda+\mu^2}) e^{-\sqrt{2\lambda +\mu^2}a_j(x,y)}\right) \frac{d\lambda}{\sqrt{2\lambda+\mu^2}} \\
& = e^{\mu (y-x)-\frac{\mu^2}{2}t} \frac{1}{2\pi i}\int_{\phi(\Gamma)} e^{\frac{\xi^2}{2} t} \frac{e^{-\xi|y-x|} }{\xi+\beta\mu} \left(\sum_{j=1}^2 c_j(\mu,y;\xi) e^{- \xi a_j(x,y)}\right) d \xi.
\end{split}\]

\begin{figure}[H]
\begin{tikzpicture}

\path(0,-3.5) node{a)};
\draw[-stealth] (-1,0) -- (3,0) node[above left]{$\mathbb{R}$};
\draw[-stealth] (0,-3) -- (0,3) node[below right]{$i \mathbb{R}$};

\draw[red] (-1,0) -- (0,0);
\draw[red,dashed,thick] (-2.5,0) -- (-1,0);
\fill[red] (0,0) circle(0.05);
\draw (-1.05,-0.1) -- (-1,-0.1) -- (-1,0.1) -- (-1.05,0.1);

\draw[blue] (-2.5,-1) -- node[above]{$\Gamma$} (0,-1) arc(-90:90:1)  -- (-2.5,1);
\draw[blue,->] (-2,-1) -- (-1,-1);
\draw[blue,->] (-1.5,1) -- (-2,1);

\draw[green] (0.5,-3) ..controls(0.6,-2.5) and (0.7,-2.1).. (1.4,-1.4) arc(-45:45:2) ..controls(0.7,2.1) and (0.6,2.5).. node[above right]{$\phi(\Gamma)$} (0.5,3);
\draw[green,->] (2,0) arc(0:5:2);

\end{tikzpicture} \hspace{0.5cm}
\begin{tikzpicture}

\path(0,-3.5) node{b1)};
\draw[-stealth] (-1.5,0) -- (3,0) node[above left]{$\mathbb{R}$};
\draw[-stealth] (0,-3) -- (0,3) node[above left]{$\boxed{\beta\mu>0}$};

\draw[green] (0.5,-3) ..controls(0.6,-2.5) and (0.7,-2.1).. node[below right]{$\phi(\Gamma)$} (1.4,-1.4) arc(-45:45:2) ..controls(0.7,2.1) and (0.6,2.5).. (0.5,3);

\draw[magenta] (0,2.25) node[left]{$M'=(0,u)$} -- node[below]{$\rho_M$} (0.73,2.25) node[right]{$M=(\ell(u),u)$};
\fill[magenta] (0,2.25) circle (0.05) (0.73,2.25) circle (0.05);
\draw[magenta] (0,-2.25) node[left]{$-M'$} -- (0.73,-2.25);
\fill[magenta] (0,-2.25) circle (0.05);

\fill[red] (1.3,0) circle(0.06) node[below right]{$|\mu|$};
\draw[red] (0,0) -- (1.3,0);
\draw[red] (0,0) circle(0.06);

\fill (-0.55,-0.05) node[below left]{$-\beta \mu$} rectangle (-0.65,0.05) ;

\end{tikzpicture} \hspace{0.5cm}
\begin{tikzpicture}

\path(0,-3.5) node{b2)};
\draw[-stealth] (-1,0) -- (3,0) node[above left]{$\mathbb{R}$};
\draw[-stealth] (0,-3) -- (0,3) node[above left]{$\boxed{\beta\mu<0}$};

\draw[green] (0.5,-3) ..controls(0.6,-2.5) and (0.7,-2.1).. node[below right]{$\phi'$} (1.4,-1.4) arc(-45:45:2) ..controls(0.7,2.1) and (0.6,2.5).. (0.5,3);
\draw[blue,thick] (1.4,1.4) arc(135:225:2) node[right]{$\gamma$};
\draw[blue,thick] (1.4,-1.4) arc(-45:45:2);
\draw[green,dashed,thick] (1.4,1.4) arc(135:225:2);

\draw[magenta] (0,2.25) node[left]{$M'$} -- node[below]{$\rho_M$} (0.73,2.25) node[right]{$M$};
\fill[magenta] (0,2.25) circle (0.05) (0.73,2.25) circle (0.05);
\draw[magenta] (0,-2.25) node[left]{$-M'$} -- (0.73,-2.25);
\fill[magenta] (0,-2.25) circle (0.05);

\fill[red] (1.7,0) circle(0.06) node[below right]{$|\mu|$};
\draw[red] (0,0) -- (1.7,0);
\draw[red] (0,0) circle(0.06);

\fill[red] (1.2,-0.1) node[below]{$-\beta \mu$} rectangle (1.4,0.1) ;

\end{tikzpicture}
\vspace{0.5cm}
\captionsetup{singlelinecheck=off}
\caption[contourdrift]{\begin{itemize}
\item[a)] The picture shows the green image of the blue contour $\Gamma$ under $\phi$. The red line $(-\infty,0]$ contains the spectrum of the operator $(L,\mathcal{D}(L))$. The dashed line $(-\infty,\frac{\mu^2}2]$ is the complement of the domain of $\phi$.
\item[b1)] Case $\beta\mu>0$: the figure represents the magenta segment $\rho_M$ connecting the unique point $M$ in $\phi(\Gamma)$ with imaginary part $u$ to its projection on the imaginary line.
The red segment $(0,|\mu|]$ is the image under $\phi$ of $(-\frac{\mu^2}{2},0]$.
\item[b2)] Case $\beta \mu<0$: the curve $\phi(\Gamma)$ is decomposed as the union of a green curve $\phi'$ that avoids the unique pole $-\beta\mu\in (0,|\mu|]$ and the blue cycle $\gamma$ containing it. The segment $\rho_M$ connects in this case the unique point $M$ in $\phi'$ with imaginary part $u$ to its projection on the imaginary line.
\end{itemize}}\label{contourdrift}
\end{figure}

If $\beta \mu >0$ the integrand is holomorphic on the region between the contour $\phi(\Gamma)$ and the imaginary line. If  $\beta\mu < 0$ the integrand has exactly one pole of order one in $\xi=-\beta \mu$. We then decompose the curve $\phi(\Gamma)$ as the union of a curve $\phi'$ and $\gamma$, where $\gamma$ is a loop around the pole and $\phi'$ avoids the pole.

Respectively $\phi(\Gamma)$ and $\phi'$ can be deformed to the imaginary line (through $H:[0,1]\times \mathbb{R}\to \mathbb{R}^2$ given by $H(t,u)=M'(1-t)+t M$), if the analogous of Lemma \ref{smallintegral} is satisfied:
\begin{lemma}\label{smallintegralbis}
\[\lim_{|u| \to +\infty}  \int_{\rho_M} e^{\frac{\xi^2}2 t} \ \frac{e^{-\xi|y-x|} }{\xi+\beta\mu} \left(\sum_{j=1}^2 c_j(\mu,y;\xi) e^{- \xi a_j(x,y)}\right) d \xi =0, \]
where $\rho_M$ is the segment in the figures connecting the points $M'$ and $M$.
\end{lemma}

\begin{proof}
First of all notice that this integral is equal to
\[ I_u :=  \int_{\rho_M} e^{\frac{\xi^2}2 t} \ \frac{e^{-\xi|y-x|} }{\xi+\beta\mu} \left(\sum_{j=1}^2 c_j(\mu,y;\xi) e^{- \xi a_j(x,y)}\right) d \xi = \int_{\rho_M} e^{\frac{\xi^2}2 t} \ e^{-\xi|y-x|} \left( 1+ \frac{c_2(\mu,y;\xi)}{\xi+\beta\mu} e^{- \xi a_2(x,y)}\right) d \xi,\]
with $\ell(u) :=| M-M' |$ and $\lim_{u\to\infty} \ell(u)=0$. Let us consider the following parametrization of the segment $\rho(M)=\left\{ M'+ v (1,0); \ v\in (0,\ell(u))\right\}$, then:
\[ |I_u| \leq \int_0^{\ell(u)} e^{-\frac{(u^2 -v^2)}{2} t}  e^{-v|y-x|}  \left(1+\left| \frac{c_2(\mu,y;v+ i u)}{v+ i u+\beta\mu} \right|e^{- v \, a_2(x,y)}\right)   dv.\]
For $u$ large enough
\[\left| \frac{c_2(\mu,y; v+ i u)}{v+i u+\beta\mu} \right| = |\beta|\sqrt{\frac{\left(v \left(2\mathbbm{1}_{[z_1,+\infty)}(y)-1\right)-\mu\right)^2+u^2}{(v+\beta\mu)^2+ u^2}}\leq 1,\]
therefore 
\[ |I_u| \leq \int_0^{\ell(u)} e^{-\frac{(u^2 -v^2)}{2} t}  e^{-v|y-x|}  \left(1+e^{- v \, a_2(x,y)}\right)   dv,\]
that converges to zero if $\left|u\right|$ goes to infinity.
\end{proof}

We compute the integral on the loop through the method of residues:
\begin{equation}  \label{jthintegral}
\begin{split}  p^{(\beta)}_\mu(t,x,y)& =   e^{\mu (y-x)-\frac{\mu^2}{2}t} \frac{1}{2\pi i}\int_{\phi'} e^{\frac{\xi^2}{2} t} \frac{e^{-\xi|y-x|} }{\xi+\beta\mu} \left(\sum_{j=1}^2 c_j(\mu,y;\xi) e^{- \xi a_j(x,y)}\right) d \xi\\
&  \quad + e^{\mu (y-x)-\frac{\mu^2}{2}t} \frac{1}{2\pi i}\int_{\gamma} e^{\frac{\xi^2}{2} t} \frac{e^{-\xi|y-x|} }{\xi+\beta\mu} \left(\sum_{j=1}^2 c_j(\mu,y;\xi) e^{- \xi a_j(x,y)}\right) \mathbbm{1}_{\mathbb{R}^-}(\beta\mu)\\ 
& = e^{\mu (y-x)-\frac{\mu^2}{2}t}   \sum_{j=1}^2 \frac{1}{2\pi}\int_{\mathbb{R}} e^{-\frac{w^2}{2} t} \frac{ 1 }{i w+\beta\mu} c_j(\mu,y;i w) e^{- i w \left(a_j(x,y)+|x-y|\right)} d w \\
& \quad \color{blue}\underbrace{\color{black}-\beta\mu \left(1+\beta\, \left(2\mathbbm{1}_{[z_1,+\infty)}(y)-1\right)\right)\,  e^{\mu (y-x)-\frac{\mu^2}{2}t}e^{\frac{\beta^2 \mu^2}{2} t}  e^{\beta\mu (|y-z_1|+|x-z_1|)} \mathbbm{1}_{\mathbb{R}^-}(\beta\mu).}_{(*)}
\end{split}
\end{equation}
The last equality is obtained by shrinking $\int_{\phi '} \to \int_{i \mathbb{R}}$ and changing variable $\xi=i w$.

\subsubsection*{C. The transition density as a sum of Fourier transforms}
We interpret each of the two integrals in the last equality of equation (\ref{jthintegral}) as the Fourier transform computed at the value $(a_j(x,y)+|x-y|)$ of the function 
\[ w\mapsto  \frac{ e^{-\frac{w^2}{2} t} }{i w+\beta\mu} c_j(\mu,y;i w)=\begin{cases}
 e^{-\frac{w^2}{2}t} &  \text{if } j=1 \\
\beta \left( \left(2\mathbbm{1}_{[z_1,+\infty)}(y)-1\right) + \mu \left({1+\beta \, \left(2\mathbbm{1}_{[z_1,+\infty)}(y)-1\right)}\right)\frac{ i}{w - i \beta\mu}\right) e^{-\frac{w^2}{2}t} &  \text{if }j=2, \\
\end{cases} \]
where the Fourier transform of $f$ is $\mathcal{F}(f)(\omega)=\hat{f}(\omega)=\frac{1}{\sqrt{2\pi}} \int_\mathbb{R} e^{-i\omega y} f(y) \, dy$.
In both cases, these functions are integrable in $w$, 
so the transition density can now be written as
\begin{equation}\label{Fourier1skewterms}
\begin{split}  p^{(\beta)}_\mu(t,x,y) & =  \frac1{\sqrt{2 \pi} }e^{\mu (y-x)-\frac{\mu^2}{2}t}   \sum_{j=1}^2 \mathcal{F}\left( e^{-\frac{w^2}{2} t} \frac{ 1 }{i w+\beta\mu} c_j(\mu,y;i w) \right)\left(a_j(x,y)+|x-y|\right) \\
& \quad -\beta\mu \, \left( 1+\beta \left(2\mathbbm{1}_{[z_1,+\infty)}(y)-1\right)\right) \,  e^{\mu (y-x)-\frac{\mu^2}{2}t}e^{\frac{\beta^2 \mu^2}{2} t}  e^{\beta\mu (|y-z_1|+|x-z_1|)} \mathbbm{1}_{\mathbb{R}^-}(\beta \mu). 
\end{split}
\end{equation}
We can assume that $\beta \mu \neq 0$ because if $\beta=0$ we get the simple Brownian motion with drift without skew, and if $\mu=0$ we get the $\beta$-SBM whose transition density is already known (see for example \cite{BS}).

\begin{lemma} \label{Fourier1}
If $a \in \mathbb{R}^*$, then
\[
\mathcal{F}\left(\frac{1}{w- i a}\right)(\omega)= i \sqrt{2 \pi} \, \left(2\mathbbm{1}_{\mathbb{R}^+}(a)-1\right) \, e^{a \omega} \, \mathbbm{1}_{\mathbb{R}^-} \left(a \,\omega \right).\]
\end{lemma}
\begin{proof} 
It is true since
\[\frac{1}{w- i a}=\mathcal{F}^{-1}\left( i \sqrt{2 \pi} \, \left(2\mathbbm{1}_{\mathbb{R}^+}(a)-1\right) \, e^{a \omega} \, \mathbbm{1}_{\mathbb{R}^-} \left(a \,\omega \right)\right)(w)=\frac{1}{\sqrt{2\pi}}\int_{\mathbb{R}} \left( i \sqrt{2 \pi} \, \left(2\mathbbm{1}_{\mathbb{R}^+}(a)-1\right) \, e^{a \omega} \, \mathbbm{1}_{\mathbb{R}^-} \left(a \,\omega \right)\right)\, e^{i \, \omega \, w} \, d\omega.\]
Notice that $\frac{1}{w-ia}$ is not integrable but $i \sqrt{2 \pi} \, \left(2\mathbbm{1}_{\mathbb{R}^+}(a)-1\right) \, e^{a \omega} \, \mathbbm{1}_{\mathbb{R}^-} \left(a \,\omega\right)$ is integrable.
\end{proof}

Using $\mathcal{F}\left(e^{-\frac{w^2}{2}t}\right)(\omega)=\frac1{\sqrt{t}}e^{-\frac{\omega^2}{2 t}}$ and Lemma \ref{Fourier1}, we get
\[ \begin{split}
\mathcal{F} & \left( e^{-\frac{w^2}{2} t} \frac{  c_2(\mu,y;i w) }{i w+\beta\mu}\right)(\omega)=\\
& \frac{\beta}{\sqrt{t}} \left(2\mathbbm{1}_{[z_1,+\infty)}(y)-1\right) e^{-\frac{\omega^2}{2 t}} - \frac1{\sqrt{t}} |\beta\mu|  \left( 1+\beta \left(2\mathbbm{1}_{[z_1,+\infty)}(y)-1\right)\right)\cdot \left(e^{w \, \beta\mu}\mathbbm{1}_{\mathbb{R}^-}\left(\beta \mu w\right) * e^{-\frac{w^2}{2 t}}\right)(\omega).
\end{split}\]
We compute the convolution as
\[\begin{split}
\left(e^{w \, \beta\mu}\mathbbm{1}_{\mathbb{R}^-}\left( \beta \mu w\right)*  e^{-\frac{w^2}{2 t}}\right)(\omega) 
 & =\sqrt{t} e^{\frac{(\beta\mu)^2}{2}t+\beta\mu \omega} \sqrt{2 \pi} \left(\underbrace{\mathbbm{1}_{\mathbb{R}^-}(\beta\mu)}_{\color{blue}(**)} + \left(2\mathbbm{1}_{\mathbb{R}^+}(\beta\mu)-1\right)\Phi^c \left(\frac{\omega}{\sqrt{t}}+\beta\mu\sqrt{t}\right)\right).
\end{split}\]

Notice that the term $(**)$ arising from the convolution is actually opposite to the term $(*)$ in (\ref{jthintegral}) arising from the integration on the cycle ${\color{blue}\gamma}$ containing the pole. Therefore the transition density becomes
\[\begin{split}  p^{(\beta)}_\mu(t,x,y) 
& = \frac1{\sqrt{2 \pi t} }e^{\mu (y-x)-\frac{\mu^2}{2}t}  \left(e^{-\frac{\left(a_1(x,y,z_1)+|x-y|\right)^2}{2 t}} + \beta \left(2\mathbbm{1}_{[z_1,+\infty)}(y)-1\right)e^{-\frac{\left(a_2(x,y,z_1)+|x-y|\right)^2}{2 t}}\right)+\\
& -\left(1+\beta\, \left(2\mathbbm{1}_{[z_1,+\infty)}(y)-1\right)\right) \beta\mu e^{\mu (y-x)-\frac{\mu^2}{2}t}  e^{\frac{(\beta\mu)^2}{2}t+\beta\mu \left(a_2(x,y,z_1)+|x-y|\right)} \Phi^c \left(\frac{a_2(x,y,z_1)+|x-y|}{\sqrt{t}}+\beta\mu\sqrt{t}\right).
\end{split}\]
Isolating the density of the Brownian motion with drift $\mu$ without skew, we recognise the expression we wanted, which completes the proof of Proposition \ref{tdf1skewmu}:
\[\begin{split}  p^{(\beta)}_\mu(t,x,y) & = p^{(0)}_\mu(t,x,y)  \left(e^{-\frac{\left(|x-y|\right)^2-|x-y|^2}{2 t}} + \beta \left(2\mathbbm{1}_{[z_1,+\infty)}(y)-1\right)e^{-\frac{\left(|x-z_1|+|y-z_1|\right)^2-|x-y|^2}{2 t}}\right)+\\
& -  p^{(0)}_\mu(t,x,y)  \sqrt{2\pi t} \beta\mu \left( 1+\left(2\mathbbm{1}_{[z_1,+\infty)}(y)-1\right)\beta\right) e^{\beta\mu a_2(x,y,z_1)}e^{\frac{(\beta\mu t +|x-y|)^2}{2 t}} \Phi ^c \left(\frac{|x-z_1|+|y-z_1|+\beta\mu t}{\sqrt{t}}\right).
\end{split}\]

\subsubsection{The case of two barriers and drift}

In this subsection we extend the computations done in the previous one to the case of two barriers to provide the transition density for the $(\beta_1,\beta_2)$-SBM with drift. 

\begin{theorem}\label{tdf2skewdrift}
Suppose $\beta_1\mu>0$ and $\beta_2\mu>0$. The transition density of the $(\beta_1,\beta_2)$-SBM with drift decomposes as
\[ p^{(\beta_1,\beta_2)}_\mu(t,x,y)=  p^{(0,0)}_\mu(t,x,y)  v^{(\beta_1,\beta_2)}_\mu(t,x,y)\]
where the function $v^{(\beta_1,\beta_2)}_\mu$ is given by a series of Fourier transforms.
If $\beta_1\neq \beta_2$,
\begin{equation} \label{d2skewmu}
v^{(\beta_1,\beta_2)}_\mu(t,x,y) = e^{\frac{|x-y|^2}{2t}} \sum_{k=0}^{\infty}  \frac{(-\beta_1 \beta_2)^k}{\beta_1-\beta_2}  \sum_{n=0}^k \sum_{m=0}^k \frac{  (-1)^m (2k-n)!}{(k-n)! (k-m)! n! m! }\frac{(\mu\sqrt{t})^{n+1-2m}}{(\beta_1-\beta_2)^{2k-n}}\sum_{j=1}^4 \sum_{h=0}^2 \frac{ c_{j,2-h}(y)}{(\mu\sqrt{t})^h} \mathscr{F}_{m,n}^h(\omega_{j,k});
\end{equation}
and if $\beta_1=\beta_2$,
\begin{equation} \label{d2skewmusamebeta}
v^{(\beta_1,\beta_1)}_\mu(t,x,y) = e^{\frac{|x-y|^2}{2t}} \sum_{k=0}^{\infty}  \frac{\beta_1^{2 k}}{(2 k+1)!} \sum_{j=1}^4 \sum_{h=0}^2 \sum_{m=0}^k {k \choose m} (-1)^{m+1} (\mu\sqrt{t})^{2(k-m)+2-h}  c_{j,2-h}(y) \mathscr{G}_{m,2k+1}^h(\omega_{j,k},\beta_1\mu\sqrt{t}),
\end{equation}
where $\omega_{j,k}:=\frac{a_j(x,y)+2 z k +|y-x|}{\sqrt t}$, $z:=z_2-z_1$ and $a_j(x,y)$ and $c_{j,h}(y)$ are defined in Lemma \ref{Green2skew}.
\[\mathscr{F}_{m,n}^h(\omega):=\mathscr{G}_{m,n}^h(\omega,\beta_2\mu \sqrt{t})-(-1)^n \mathscr{G}_{m,n}^h(\omega,\beta_1\mu \sqrt{t})\]
and for $A\in \left\{\beta_1 \mu \sqrt{t}, \beta_2 \mu \sqrt{t}\right\}$,

\[\begin{split}
\mathscr{G}_{m,n}^h(\omega,A)  &= (2m+h)! \sum_{\ell=0}^{m+\lfloor {\frac{h}{2}}\rfloor} \frac{(-1)^{\ell+h} }{2^\ell}\frac{1}{\ell! (2(m-\ell)+h)!} \ S_{m,n,l}^h(\omega,A)
\end{split}\]
where
\[\begin{split}
 S_{m,n,l}^h(\omega,A) 
& = \sum_{r=0}^n \sum_{s=0}^{2(m-\ell)+h} {n \choose r} {2(m-\ell)+h \choose s} (\omega+A)^{n-r} A^{2(m-\ell)+h-s} J_{r+s}(\omega,A),
\end{split}\]
and 
\[\begin{split} J_{q}(\omega,A) 
& :=
\begin{cases}
\sqrt{2 \pi} e^{\frac{A^2}{2}+A\omega} \Phi^c(\omega+A) 
& q=0,\\
-e^{-\frac{\omega^2}{2}} & q=1, \\
J_0(\omega, A) (q-1)!! - J_1(\omega,A) \sum_{a=0}^{\frac{q}{2}-1} (\omega+A)^{q-2a-1} \frac{(q-1)!!}{(q-2a-1)!!}  & q\geq 2 \text{ even, }\\
J_1(\omega,A) \sum_{a=0}^{\frac{q-1}{2}} (\omega+A)^{(q-1-2a)} 2^a \frac{(\frac{q-1}{2})!}{(\frac{q-1}{2}-a)!} & q \geq 3 \text{ odd }.
\end{cases} \end{split}\]
\end{theorem}

The proof of the theorem is based on the following four lemmas.

\begin{lemma}\label{Green2skew}
The Green functions, defined in Lemma \ref{Greenfunctions} satisfy
\[G(x,y;w) = \frac{1}{w}  e^{\mu(y-x)}  e^{- w |x-y|} \frac{ \sum_{j=1}^4 c_j(\mu,y;w)e^{-w a_j(x,y)}}{\beta_1 \beta_2 e^{-2 w z} (w^2-\mu^2)  +   (w+\beta_1 \mu)(w+\beta_2 \mu)}.\]
where $w:=\sqrt{2\lambda+\mu^2}$ and $z:=z_2-z_1$ is the distance between the barriers. The functions $a_j(x,y)$ are non negative, in particular they are

\[\begin{cases}
a_1(x,y)\equiv 0\\
a_2(x,y)=|y-z_1|+|x-z_1|-|y-x|\\
a_3(x,y)=|y-z_2|+|y-z_2|-|y-x|\\
a_4(x,y)=2\left(z_2-max(x,y,z_1)\right)^+ + 2 \left(min(x,y,z_2)-z_1\right)^+
\end{cases}\]
and 
$c_j(\mu,y;w)=w^2 c_{j,0}(y)+w \mu c_{j,1}(y)+ \mu^2 c_{j,2}(y)$ where
\[\begin{cases}
c_{1,0}(y)=1, 											\\
c_{2,0}(y)=\left(2\mathbbm{1}_{[z_1,+\infty)}(y)-1\right)\beta_1 							\\
c_{3,0}(y)=\left(2\mathbbm{1}_{[z_2,+\infty)}(y)-1\right)\beta_2 							\\
c_{4,0}(y)=\left(1-2\mathbbm{1}_{[z_1,z_2)}(y)\right)\beta_1\beta_2 
\end{cases},
\begin{cases}
c_{1,1}(y)=\beta_1+\beta_2		\\
c_{2,1}(y)=-\beta_1-c_{4,0}(y)		\\
c_{3,1}(y)=-\beta_2+c_{4,0}(y)		\\
c_{4,1}(y)=0
\end{cases},
\begin{cases}
c_{1,2}(y)=\beta_1\beta_2 \\
c_{2,2}(y)=\beta_1 c_{3,0}(y) \\
c_{3,2}(y)=-\beta_2 c_{2,0}(y) \\
c_{4,2}(y)=-c_{4,0}(y).
\end{cases}\]
\end{lemma}
\begin{proof} Analogously of the proof provided in subsection \ref{sectionGreen2skew}.
\end{proof}

\begin{lemma}[Partial fractional decomposition] \label{PFD}
Let $a,b\in \mathbb{R}^*$, $a\neq b$, then
\[\frac{1}{(w-i a)^{k+1} (w-ib)^{k+1}}=i\sum_{j=0}^{k} \frac{1}{(a-b)^{2 k+1-j}} \binom{2k-j}{k-j} \left(\frac{i^j}{(w-ib)^{j+1}}- (-1)^j\frac{i^j}{(w-ia)^{j+1}}\right).\]
\end{lemma}
\begin{proof}
The function $f(x)=\frac{1}{(w-i a)^{k+1}(w-i b)^{k+1}}$ is a rational function  with two poles $x_1=i a, x_2= i b$ of order $k+1$. We followed a standard method for computing the decomposition: there exist coefficients $\alpha_{i,j}$ such that the function can be written as $f(x)=\sum_{i=1}^2 \sum_{j=1}^{k+1} \frac{\alpha_{i,j}}{(x-x_i)^j}$. Since the $\alpha_{i,j}$ are the residues in $x_i$ of the function $g_{i,j}(x)=(x-x_i)^{j-1}f(x)$, we computed them explicitly.
\end{proof}

\begin{lemma} \label{Fourier2}
If $a \in \mathbb{R}^*$, and $k\in \mathbb{N}$ then
\[
\mathcal{F}\left(\frac{1}{(w- i a)^{k+1}}\right)(\omega)= i^{k+1} \sqrt{2 \pi} \, \left(2\mathbbm{1}_{\mathbb{R}^+}(a)-1\right) \, \frac{(-\omega)^k}{k!} \, e^{a \omega} \, \mathbbm{1}_{\mathbb{R}^-} \left(a \,\omega \right).\]
\end{lemma}
\begin{proof} 
If $k=0$ it coincides with Lemma \ref{Fourier1}, otherwise the function $\frac{1}{(w- i a)^{k+1}}\in L^1(\mathbb{R})\cap L^2(\mathbb{R})$ and one computes its Fourier transform in $\omega$, $\frac1{\sqrt{2\pi}} \int_\mathbb{R} \frac{1}{(w- i a)^{k+1}} e^{-i \omega w}dw$, through the method of residues.
\end{proof}

\begin{lemma} \label{integraldef}
Let $q\in \mathbb{N}$. The primitive function $I_q(\cdot)$ of $v\mapsto v^{q}e^{-\frac{v^2}2}$  (resp. $\tilde{ I}_q(\alpha)=\int_{(\alpha,+\infty)} v^q \, e^{-\frac{v^2}{2}} dv$ ) satisfies
\[\begin{cases}
I_0(\alpha) := 
\sqrt{2 \pi} \Phi(\alpha)=\sqrt{2 \pi} \Phi^c(-\alpha) \quad (\text{resp. } \sqrt{2 \pi} \Phi^c(\alpha)) 
& q=0,\\
I_1(\alpha) := 
- e^{-\frac{\alpha^2}{2}} \qquad (\text{resp. } e^{-\frac{\alpha^2}{2}})& q=1, \\
I_q(\alpha)= I_0(\alpha) (q-1)!! + I_1(\alpha) \sum_{k=0}^{\frac{q}{2}-1} \alpha^{q-2k-1} \frac{(q-1)!!}{(q-2k-1)!!}  & q\geq 2 \text{ even, }\\
I_q(\alpha)=I_1(\alpha) \sum_{k=0}^{\frac{q-1}{2}} \alpha^{(q-1-2k)} 2^k \frac{(\frac{q-1}{2})!}{(\frac{q-1}{2}-k)!} & q \geq 3 \text{ odd }.
\end{cases}\]
\end{lemma}
\begin{proof} Straightforward for $q=0,q=1$, and for $q\geq 2$ one can use the integration by parts for the integral $\int v^q e^{-\frac{v^2}{2}} dv= - \int v^{q-1} \frac{d}{dv} \left(e^{-\frac{v^2}{2}}\right) dv$ and obtain the recursive formula
\[I_{q}(\alpha)=\alpha^{q-1}I_1(\alpha)+(q-2)I_{q-2}(\alpha)\]
that yields the conclusion.
\end{proof}

We just present a sketch of the proof of Theorem \ref{tdf2skewdrift}. The detailed computations will be proposed in Appendix. The ideas are similar to the proof of Proposition \ref{transitiondensity2skew} and Proposition \ref{tdf1skewmu} but even more technical and laborious.

\begin{proof}[Proof of Theorem \ref{tdf2skewdrift}]
In subsection \ref{method} we saw that the transition density of the $(\beta_1,\beta_2)$-SBM with drift $\mu$ has an integral representation as in equation \ref{representationdensity}. Lemma \ref{Green2skew} gives us the expression of the Green functions. One can make the change of variable $\phi(\lambda)=\sqrt{2 \lambda +\mu^2}$ proceeding as in Figure~ \ref{contourdrift}.a. We can show that zero is an erasable singularity for the integrand, that is also holomorphic on the entire imaginary line. Since we assumed $\beta_1\mu>0, \, \beta_2 \mu>0$, the integrand has no poles in $(0,\mu^2]$. Therefore, being in the case of Figure~\ref{contourdrift}.b1 and since an analogous of Lemma \ref{smallintegral} holds, one can deform the contour to the imaginary line. One obtains the transition density as
\[p^{(\beta_1,\beta_2)}_\mu(t,x,y)=- e^{-\frac{\mu^2}{2}t+\mu (y-x)} \frac{1}{2\pi} \int_{\mathbb{R}} e^{-\frac{w^2}{2}t}  e^{- i w |x-y|} \frac{ \sum_{j=1}^4 c_j(y,\mu;iw)e^{-i w a_j(x,y)}}{\beta_1 \beta_2 e^{-2 i w z} (w^2+\mu^2)  +   (w-i \beta_1 \mu)(w-i\beta_2 \mu)} dw.\]
If $w\neq 0$ then $\left| \beta_1\beta_2 \frac{w^2+\mu^2}{(w-i\beta_1\mu)(w-i\beta_2\mu)}\right|<1$ hence the transition density can be written as
\[\begin{split}
p^{(\beta_1,\beta_2)}_\mu= \frac{e^{-\frac{\mu^2}{2}t+\mu (y-x)}}{2\pi }\int_{\mathbb{R}}  \sum_{k=0}^{\infty}  \frac{ -(-\beta_1 \beta_2)^k(w^2+\mu^2)^k}{\left[(w-i\beta_1 \mu)(w-i\beta_2 \mu)\right]^{k+1}}e^{-\frac{w^2}{2}t} \sum_{j=1}^4 c_j(y,\mu; i w) e^{- iw (a_j(x,y,\beta)+2 z k + |x-y|)} dw,
\end{split}\]
where we can exchange integral and limit because the series of the absolute values is integrable. We now interpret the expression for $v^{(\beta_1,\beta_2)}_\mu$ as a series of Fourier transforms
\begin{equation}\label{firststepfourier}
\begin{cases}
v^{(\beta_1,\beta_2)}_\mu(t,x,y) = e^{\frac{|x-y|^2}{2t}} \sum_{k=0}^{\infty}  (-\beta_1 \beta_2)^k  \sum_{j=1}^4 F_{j,k}(\omega_{j,k}),\\
 F_{j,k}:=\mathcal{F}\left(w \mapsto  e^{-\frac{w^2}{2}} c_j(y,\mu\sqrt{t};i w) (w^2+\mu^2t)^k\cdot\frac{-1}{(w-i\beta_1 \mu\sqrt{t})^{k+1}(w-i\beta_2 \mu\sqrt{t})^{k+1}} \right)
\end{cases}
\end{equation}
The Fourier transform $F_{j,k}(\omega_{j,k})$ can be rewritten as the convolution of Fourier transforms
\begin{equation}\label{secondstepfourier}
F_{j,k}(\omega_{j,k})=\frac{1}{\sqrt{2\pi}}\mathcal{F}\left(  e^{-\frac{w^2}{2}} c_j(y,\mu\sqrt{t};i w) (w^2+\mu^2t)^k\right) * \mathcal{F}\left(\frac{-1}{(w-i\beta_1 \mu\sqrt{t})^{k+1}(w-i\beta_2 \mu\sqrt{t})^{k+1}} \right) (\omega_{j,k}).
\end{equation}
The Fourier transform $\mathcal{F}\left(\frac{-1}{(w-i\beta_1 \mu\sqrt{t})^{k+1}(w-i\beta_2 \mu\sqrt{t})^{k+1}} \right)$ is computed using Lemma \ref{Fourier2} if $\beta_1=\beta_2$, otherwise using jointly Lemma \ref{PFD} and Lemma \ref{Fourier2}.
One concludes the proof using the properties of the iterated derivatives of Gaussian densities, and introducing $J_q(\omega,A)=
e^{\frac{A^2}2+A\omega}I_q(-(\omega+A))$ (see Lemma \ref{integraldef})
. For more details see the Appendix.
\end{proof}

We assumed both $\beta_1\mu$ and $\beta_2\mu$ to be positive because, if $\beta_1\mu<0$ or $\beta_2\mu<0$, the exact computation for the density can be even more subtle. This is due to the possible presence of an additional term in the contour integral corresponding to zeros of the denominator lying the positive real semi-axis (as in the case of the $\beta_1$-SBM with drift $\mu$, see Figure~\ref{contourdrift}.b2). These cases will be treated in a incoming paper on the exact simulation of a Brownian diffusion with drift with several discontinuities \cite{DMR2}.

Another possible approach in order to solve the (\ref{traden}) could be to apply the technique used in \cite{Vee} in case of Brownian motion with drift between two barriers, but our approach seems to be more fruitful.

\subsection{Limit cases}
For particular choices of the parameters formulas (\ref{d2skewmu}) and (\ref{d2skewmusamebeta}) reduce to the more simple cases studied before.

For $\beta_2=0$, the correspondent barrier $z_2$ is completely permeable, so it is like if it disappears, hence one would expect to obtain the density of the $\beta_1$-SBM with drift.\\
Without directly substituting $\beta_2=0$ in the final expression of the transition density, one can notice in equation (\ref{firststepfourier}) that only $F_{j,0}(\omega_{j,0})$ for $j\in\left\{1,2\right\}$ do not vanish. Moreover equation (\ref{firststepfourier}) turns out to be equation (\ref{Fourier1skewterms}) with $\beta=\beta_1$ such that $\beta_1\mu>0$.

Even for $z_2\to+\infty$ one would expect to obtain the density of the $\beta_1$-SBM with drift. In fact if the second barrier is very far from the starting point of the process, at every finite time the trajectory has no way to see the latter barrier and is effected only by the reflection coefficient $\beta_1$.\\
Less heuristical and more direct would the following approach. First notice that, since $z_2\to +\infty$, $a_3(x,y), a_4(x,y)$ and $z$ go to $+\infty$ which implies $\omega_{j,k}\to \infty$ as soon as $k\neq 0$ or $j\neq \{1,2\}$. Then consider the expression for $F_{j,k}$: in equation (\ref{firststepfourier}), it is a Fourier transform of a $L^2$-function, hence it is in $L^2$. It can be shown that it admits a limit at infinity, hence this limit has to be zero. Therefore the not vanishing terms in equation \ref{firststepfourier} are again given by $j=1,2$ and $k=0$.

\section{Exact simulation}\label{exactsimulation}
To simulate exactly a process means to simulate it from its law sampling exactly from its finite dimensional distributions without approximations (beyond the machine's). 
Exact sampling of a random variable can be achieved using the rejection sampling method, introduced in \cite{Neu}.

The rejection method allows to sample from the density $h$ of a random variable $X\sim h(x)dx$ knowing how to sample another one $Y\sim g(x)dx$ if $h \leq M g$ for $M$ a finite strictly positive constant. The sample $y=Y$ is accepted as a sample of $X$ if and only if $u<\frac{h(y)}{ M g(y)}$ , where $u$ is the sample of a uniform random variable $U \sim \mathcal{U}_{[0,1]}$. Notice that $\mathbbm{1}_{\left\{U< \frac{h(Y)}{M g(Y)}\right\}}$ is a Bernoulli random variable with random parameter $\frac{h(Y)}{M g(Y)}$. Moreover the densities $h(x)$ and $g(x)$ do not need to be normalized.

In our framework, the one-dimensional projection at time $t$ of a $(\beta_1,\beta_2)$-SBM has a density whose ratio with respect to the well known transition probability density of the Brownian motion is a series, as we saw in equation (\ref{tdf2skew}). What happens if the density cannot be evaluated exactly, since it is an infinite sum?
The technique we are going to propose allows to evaluate only a finite number of terms of the series, and at the same time, to maintain the exactness of the sampling.

\subsection{Generalized rejection sampling method}
Let us introduce our method by explaining a toy example for simulating exactly a Bernoulli random variable $X\sim \mathcal{B}_p$ with unknown parameter $p\in [0,1]$. If the parameter is known, then clearly $X \overset{(d)}{=}\mathbbm{1}_{\left \{U\leq p\right\}}$, hence an exact simulation consists in sampling the uniform random variable $U\sim \mathcal{U}_{[0,1]}$ and checking if the sample is smaller (or bigger) than $p$ to decide if $X=1$ (or $X=0$).

\begin{lemma} \label{idearejection}
Suppose $p$ is an unknown parameter which is approximated by a sequence $(p_n)_n$ and the rate of convergence is at least $(\delta_n)_n$ where $(\delta_n)_n$ is a decreasing vanishing sequence (i.e. $|p-p_n|<\delta_n$). Then it is possible to simulate exactly a Bernoulli of parameter $p$ since $X:=\mathbbm{1}_{\left\{\exists n; \  |U-p_n|>\delta_n, \ U < p_n \right\}}\sim \mathcal{B}_p$ .
\end{lemma}

\begin{proof}

First of all we need to show that, a.s., there exists an $n$ such that $\left| U-p_n\right| > \delta_n$. Notice that a.s. $\left| U-p\right|>0$. Since $\delta_n\to 0$, a.s. there exist $n_0$ such that $\left| U-p\right| > 2 \delta_{n_0}$. Therefore there exist an $n$ (for monotonicity it works for $n\geq n_0$) such that a.s. $\left|  U-p_n \right|>\delta_n$.\\
Now, since 
\[\left\{U<p\right\}=\left\{\exists n\in \mathbb{N}; \  |U-p_n|>\delta_n, \ U < p_n \right\}=\bigcup_{n\in \mathbb{N}} \left\{U<p_n-\delta_n \right\} ,\]
then $p=\mathbb{P}(U<p)=\mathbb{P}(X=1)$.

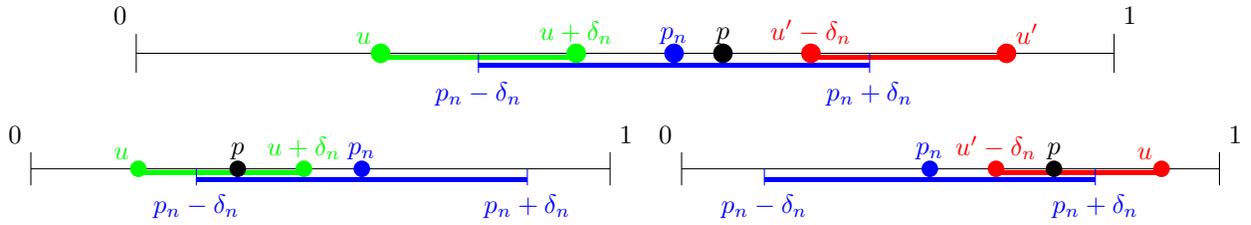
\begin{figure}[H]
\begin{center}
\begin{tikzpicture}[scale=1.3]

\draw[draw] (-5,0) -- (5,0);
\draw[draw] (-5,-0.2) -- (-5,0.2) node[above left]{$0$};
\draw[draw] (5,-0.2) -- (5,0.2) node[above right]{$1$};
\fill[blue] (-1.5,-0.1)rectangle(2.5,-0.15) ;
\fill[green] (-2.5,-0.02)rectangle(-0.5,-0.07) ;
\fill (1,0) circle (0.1) node[above]{$p$};
\fill[green] (-2.5,0) circle (0.1) node[above left]{$u$};
\fill[green] (-0.5,0) circle (0.1) node[above]{$u+\delta_n$};
\fill[red] (3.9,0) circle (0.1) node[above right]{$u'$};
\fill[red] (1.9,0) circle (0.1) node[above]{$u'-\delta_n$};
\fill[red] (1.9,-0.02)rectangle(3.9,-0.07) ;
\fill[blue] (0.5,0)  node[above]{$p_n$} circle (0.1);
\draw[blue] (-1.5,-0.2) node[below]{$p_n-\delta_n$} -- (-1.5,0) ;
\draw[blue] (2.5,-0.2) node[below]{$p_n+\delta_n$} -- (2.5,0) ;

\end{tikzpicture}\\
\begin{tikzpicture}[scale=1.1]

\draw[draw] (-3.5,0) -- (3.5,0);
\draw[draw] (-3.5,-0.2) -- (-3.5,0.2) node[above left]{$0$};
\draw[draw] (3.5,-0.2) -- (3.5,0.2) node[above right]{$1$};
\fill[blue] (-1.5,-0.1)rectangle(2.5,-0.15) ;
\fill[green] (-2.2,-0.02)rectangle(-0.2,-0.07) ;
\fill (-1,0) circle (0.1) node[above]{$p$};
\fill[green] (-2.2,0) circle (0.1) node[above left]{$u$};
\fill[green] (-0.2,0) circle (0.1) node[above]{$u+\delta_n$};
\fill[blue] (0.5,0)  node[above]{$p_n$} circle (0.1);
\draw[blue] (-1.5,-0.2) node[below]{$p_n-\delta_n$} -- (-1.5,0) ;
\draw[blue] (2.5,-0.2) node[below]{$p_n+\delta_n$} -- (2.5,0) ;
\end{tikzpicture}
\begin{tikzpicture}[scale=1.1]
\draw[draw] (-2.5,0) -- (4,0);
\draw[draw] (-2.5,-0.2) -- (-2.5,0.2) node[above left]{$0$};
\draw[draw] (4,-0.2) -- (4,0.2) node[above right]{$1$};
\fill[blue] (-1.5,-0.1)rectangle(2.5,-0.15) ;
\fill[red] (1.3,-0.02)rectangle(3.3,-0.07) ;
\fill (2,0) circle (0.1) node[above]{$p$};
\fill[red] (3.3,0) circle (0.1) node[above left]{$u$};
\fill[red] (1.3,0) circle (0.1) node[above]{$u'-\delta_n$};
\fill[blue] (0.5,0)  node[above]{$p_n$} circle (0.1);
\draw[blue] (-1.5,-0.2) node[below]{$p_n-\delta_n$} -- (-1.5,0) ;
\draw[blue] (2.5,-0.2) node[below]{$p_n+\delta_n$} -- (2.5,0) ;

\end{tikzpicture}

\end{center}
\caption{The pictures illustrate the way to sample a Bernoulli random variable $X$ of unknown parameter $p$: if $u<p$ then $X:=1$, otherwise $X:=0$. In the first image $u<p_n-\delta_n$ hence $u<p$ (resp. $u'>p_n+\delta_n$ hence $u'>p$). The second images show that, if $u<p_n-\delta_n<u+\delta_n<p_n$ (resp. $p_n<u'-\delta_n<p_n+\delta_n<u$) then $u<p$ (resp. $u'>p$) anyway.}\label{ideaalgorithm}
\end{figure}

The scheme of the algorithm then will be:
\begin{enumerate}
\item sample from $\mathcal{U}$, we obtain $u$
\item find $n$ such that $\left| u-p_n\right| > \delta_n$,
\item if $u < p_n$, then $u<p$ hence $X:=1$ otherwise $X:=0$ (see Figure~\ref{ideaalgorithm}).
\end{enumerate}
\end{proof}

This idea allows us to extend the rejection sampling method for sampling $X\sim h(x)dx$ knowing an approximation of the density $h(x)$.

\begin{theorem}[Generalized rejection sampling method] \label{generalizedrejection}
Assume one knows how to sample the random variable Y with (unnormalized) density $g(x)$. Then one can sample the random variable $X$ with (unnormalized) density $h(x)$ under the following assumptions:
\begin{enumerate}[(i)]
\item the ratio between the functions $g$ and $h$ is bounded:
\[ \exists M>0 \text{ such that } 0<f(y):=\frac1M\frac{h(y)}{ g(y)}\leq 1 \text{ for all } y\in \mathbb{R};\]
\item there exists a sequence of explicitly computable functions $(f_n)_n$ converging to $f$ at a decreasing explicitly computable rate $(\delta_n)_n$.
\end{enumerate}
Then $X\sim \left(Y| \, \exists n; \ U<f_n(Y)-\delta_n \right)$ i.e. an exact simulation is possible.
\end{theorem}

\begin{proof}

It is well known from the standard rejection sampling that $X\sim \left(Y|{U}<f(Y)\right)$ (see for example \cite{Ross}). Lemma \ref{idearejection} ensures that we can simulate exactly without knowing $f(Y)$ with complete accuracy. The acceptability of the draw $y=Y$ as a sample from $X$ is a Bernoulli with parameter $f(y)$ and we can compute explicitly a sequence converging to this quantity $f_n(y)(=:p_n)$ and its rate of convergence $(\delta_n(y))_n$. 
Thus the rejection sampling scheme based on Lemma \ref{idearejection} is the following 
\begin{enumerate}
\item sample $u$ from a uniform random variable $U\sim \mathcal{U}_{[0,1]}$,
\item sample from the density $g$: we get $y=Y$,
\item take $y$ as a sample of $X$ if $u<f(y)$, otherwise reject and start again. More precisely:
\begin{itemize}
\item[3a.] find $n$ such that $|u-f_n(y)|>\delta_n$,
\item[3b.] check whether $f_n(y)>u$,
\item[3c.] if yes accept $X=y$, if not reject it.
\end{itemize}
\end{enumerate}
\end{proof}

\subsection{Sampling from the density of the $(\beta_1,\beta_2)$-SBM}

We now apply Theorem \ref{generalizedrejection} for sampling from the density at time $t$ of the $(\beta_1,\beta_2)$-SBM starting at $x$. We already noticed in Proposition \ref{transitiondensity2skew} that its density is absolutely continuous with respect to the one of the Brownian motion $p^{(0,0)}(t,x,y)$ with ratio 
\begin{equation} \label{vbeta}
v^{(\beta_1,\beta_2)}(t,x,y)=\sum_{k=0}^{\infty}  (-\beta_1 \beta_2)^k  \sum_{j=1}^4 c_j(y,\beta_1,\beta_2)e^{ -\frac{\left(a_j(x,y)+2 z k\right)^2}{2 t}} e^{-|x-y|\frac{a_j(x,y)+2 z k}{t}}.
\end{equation}
Since we can only evaluate the sum of the series $v^{(\beta_1,\beta_2)}(t,x,y)$ with some error, we check if the hypothesis $(i)$ and $(ii)$ of Theorem \ref{generalizedrejection} are satisfied.

\begin{lemma} \label{boundv}
There exists an upper bound for $v^{(\beta_1,\beta_2)}(t,x,y)$ uniform in $x$ and $y$:  
\[ \sup_{x,y\in\mathbb{R}} \left| v^{(\beta_1,\beta_2)}(t,x,y)\right| \leq \overline{v}:=\frac{(1+|\beta_1|)(1+|\beta_2|)}{1-|\beta_1\beta_2|}.\]
\end{lemma}
\begin{proof}
\[|v^{(\beta_1,\beta_2)}(t,x,y)|\leq \left(\sum_{k=0}^{\infty} |\beta_1\beta_2|^k \right)\left(\sum_{j=1}^4 |c_j(y,\beta_1,\beta_2)|\right)=\frac{\sum_{j=1}^4 |c_j(y,\beta_1,\beta_2)|}{1-|\beta_1\beta_2|}=\frac{(1+|\beta_1|)(1+|\beta_2|)}{1-|\beta_1\beta_2|}=:\overline{v}.\]
\end{proof}

We denote the truncated series at the first $N$ terms by
\[v^{(\beta_1,\beta_2)}_N(t,x,y):=\sum_{k=0}^{N}  (-\beta_1 \beta_2)^k  \sum_{j=1}^4 c_j(y,\beta_1,\beta_2) e^{ -\frac{\left(a_j(x,y)+2 z k\right)^2}{2 t}}  e^{-|x-y|\frac{a_j(x,y)+2 z k}{t}},\]
and the rest by $R^N v^{(\beta_1,\beta_2)}(t,x,y)=v^{(\beta_1,\beta_2)}(t,x,y)-v_N^{(\beta_1,\beta_2)}(t,x,y)$.

\begin{lemma} \label{boundrest}
The rest of the truncated series is bounded uniformly in $x,y$: 
\[|R^N v^{(\beta_1,\beta_2)}(t,x,y)|\leq \overline{v} \ |\beta_1 \beta_2|^{N+1}.\]
\end{lemma}
\begin{proof}
\[|R^N v^{(\beta_1,\beta_2)} (t,x,y) | \leq \left(\sum_{k=N+1}^{\infty}  |\beta_1 \beta_2|^k\right) \left( \sum_{j=1}^4 |c_j(y,\beta_1,\beta_2) |\right) =\overline{v} \ |\beta_1\beta_2|^{N+1}.\]
\end{proof}

We can apply Theorem \ref{generalizedrejection} with
\[f_n (y):= \frac1{\overline{v}} \ v^{(\beta_1,\beta_2)}_n(t,x,y) \quad \text{ and } \quad \delta_n:=|\beta_1\beta_2|^{n+1}.\]
We are then able to sample from the density $y\mapsto p^{(\beta_1,\beta_2)}(t,x,y)$ in equation (\ref{tdf2skew}) through the generalized rejection sampling algorithm and therefore we are able to simulate exactly the Markov process $(\beta_1,\beta_2)$-SBM (for example see Figure~\ref{pathexactsimulation}).

To increase the efficiency of the rejection algorithm, we apply the following principle: assume we have just computed $\left|f_N^{\beta,t,x}(y)-u \right|$ and noticed that it is smaller than $\delta_N$, we then take the first index $\hat{N}$ greater than the quantity ${{\left(\log \delta_N\right)^{-1}}\log \left|f_N^{\beta,t,x}(y)-u\right|}$. Moreover it is better to fix an integer $N_{max}$ in order to stop the algorithm in case it does not find the desired conditions $3a$. in Theorem \ref{generalizedrejection}. This index $N_{max}$ should be such that the rest of the series is sufficiently small for considering the truncated sum as a good approximation (due to Lemma \ref{boundrest} an upper bound for the error is $\overline{v} \ |\beta_1\beta_2|^{N_{max}}$). In any case the simulation turns out to be always exact (that is the acceptance or rejection is obtained for an index smaller than $N_{max}$) if $|\beta_1\beta_2|$ is not too close to 1. In that case we may increase the index $N_{max}$ in such a way that $\delta_{N_{max}}$ is small.

\begin{figure}[H]\centering\begin{minipage}{0.45\textwidth}
\centering
\includegraphics[width=9.cm]{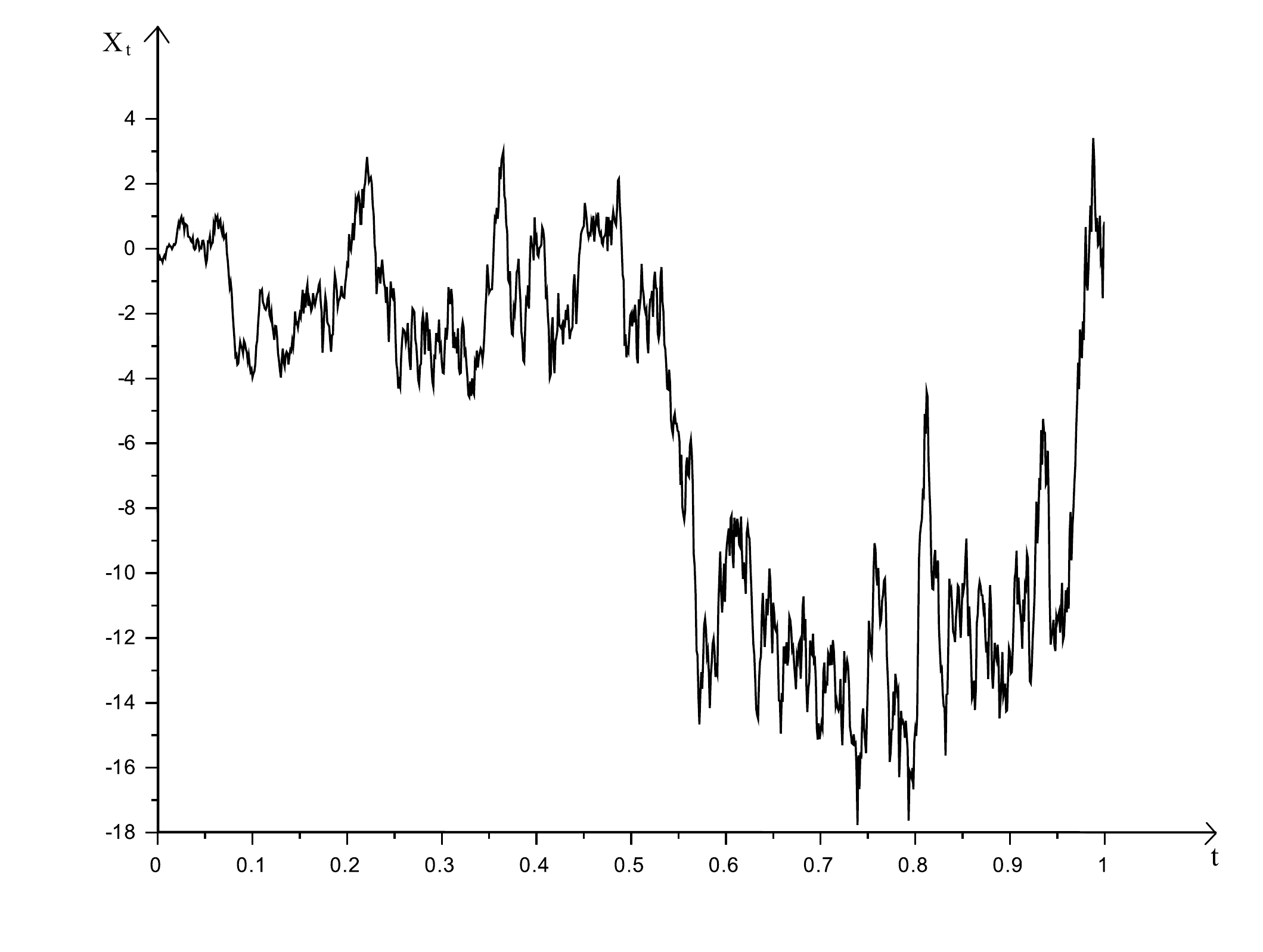}
\caption{Exact simulation of a path of the $(0.7,-0.2)$-SBM starting at time $0$ in $x=-0.3$. The barriers are $z_1=0$ and $z_2=1$.}\label{pathexactsimulation}
\end{minipage}
\hfill
\begin{minipage}{0.45\textwidth}
\includegraphics[width=9cm]{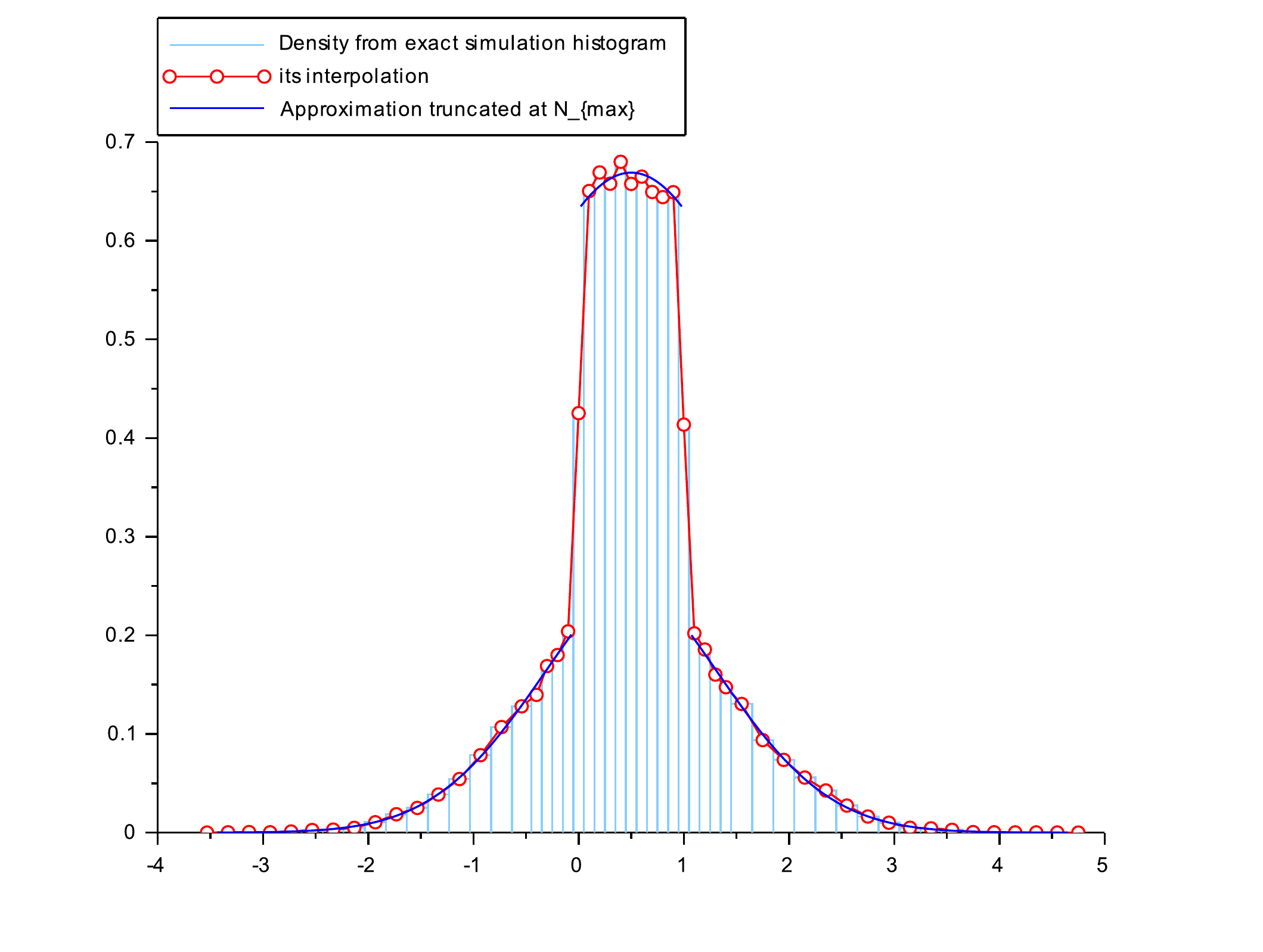}
\caption{Comparison between the function $y \mapsto p^{(\frac12,-\frac12)}(1,0.5,y)$ obtained from 50000 exact simulations through generalized rejection sampling method and the plot of its truncated version at the tenth term ($N_{max}=10$). The barriers are $z_1=0$ and $z_2=1$.}\label{symmetricdensity}
\end{minipage}
\end{figure}

Let us compare now the approximation of the density $y\mapsto p^{(\beta_1,\beta_2)}(t,x,y)$ in equation (\ref{tdf2skew}) obtained truncating the series at the $N_{max}$-th term and an histogram of a large number of \emph{exact} samples from the untruncated density computed through the generalized rejection sampling method. For simplicity, we always take time $t=1$, starting point $x=0.5$ and we assume that the barriers are fixed in $z_1=0$ and $z_2=1$.

We represent in Figure~\ref{symmetricdensity}, as typical situation, the function $y \mapsto p^{(\frac12,-\frac12)}(1,0.5,y)$. In this case, $100 \%$ of the 50000 simulations are exact. The average number of terms of the series that are necessary in order to decide if to accept or reject the simulations is smaller than $2$ ($1.6$). From now on we will denote this number as $N_{rej}$.\\
The transition density in this case is mainly concentrated inside of the interval between the barriers $(z_1,z_2)$ since $\beta_1>0$ and $\beta_2<0$. Choosing $N_{max}=10$ the truncated series differs from the untruncated one at most of $\overline{v} \, |\beta_1\beta_2|^{11}\sim 6 \cdot 10^{-7}$.

\begin{figure}[H]\centering\begin{minipage}{0.45\textwidth}
\centering
\includegraphics[width=9cm]{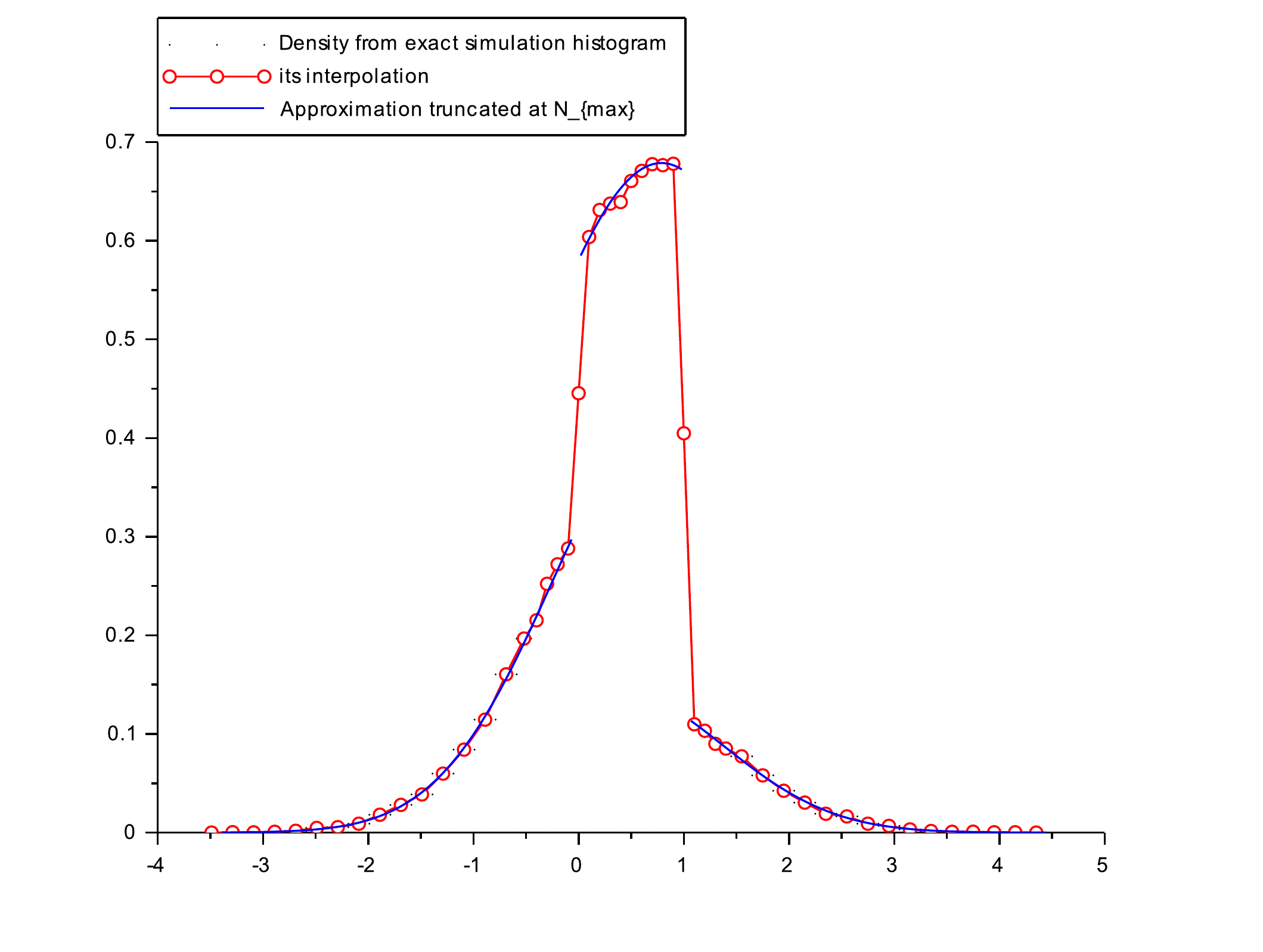}
\caption{Comparison between the function $y \mapsto p^{(0.3,-0.7)}(1,0.5,y)$ obtained from 50000 exact simulations through generalized rejection sampling method and the plot of its truncated version at the tenth term ($N_{max}=10$). $\delta_{10}=3.5 \cdot 10^{-8}$ and $Bv=2.8$. The barriers are $z_1=0$ and $z_2=1$. The average acceptance number is $N_{rej}=1.28$.}\label{nonsymmetricdensity}
\end{minipage}
\hfill
\begin{minipage}{0.45\textwidth}
\includegraphics[width=9cm]{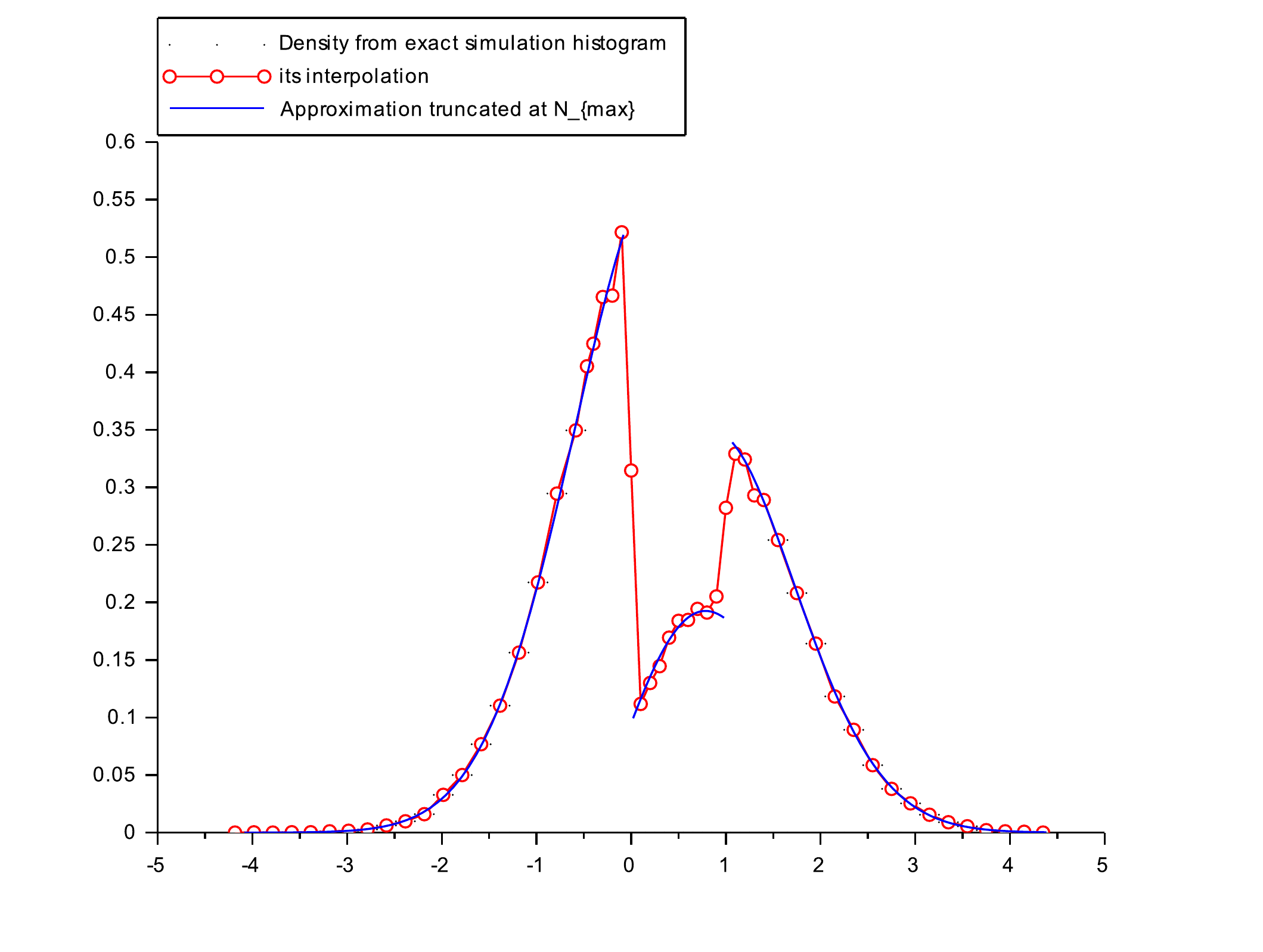}
\caption{
Comparison between the function $y \mapsto p^{(-0.7,0.3)}(1,0.5,y)$ obtained from 50000 exact simulations through generalized rejection sampling method and the plot of its truncated version at $N_{max}=10$. The barriers are $z_1=0$ and $z_2=1$. The average acceptance number $N_{rej}$ is $1.27.$}\label{nonsymmetricout}
\end{minipage}
\end{figure}

In Figure~\ref{nonsymmetricdensity} and \ref{nonsymmetricout} we propose skewness parameters with different absolute values and pointing respectively inward and outward. All our simulations are exact and $N_{rej}\sim 1.3$ is low as expected. In these cases $\delta_{n}=0.21^{n+1}$ and $\overline{v}=2.8$. We can observe in Figure~\ref{nonsymmetricdensity} that the process tends to stay between the barriers because when it reaches the barrier $z_1$ it has probability $\frac{1+\beta_1}2=0.65$ to be reflected to this region and when it reaches $z_2$ the probability is $\frac{1-\beta_2}{2}=0.85$. If the process leaves $(z_1,z_2)$, then the probability to be before $z_1$ is larger than to be after $z_2$ because $1-\beta_1 > 1+\beta_2$.\\
In Figure~\ref{nonsymmetricout} the parameters $\beta_1=-0.7$ and $\beta_2=0.3$ induce that the process is more likely to be outside the region between the barriers because it is reflected outside this region with probability $\frac{1-\beta_1}2=0.85$ in $z_1$ and with probability $0.65$ in $z_2$.

\begin{figure}[H]\centering\begin{minipage}{0.45\textwidth}
\centering
\includegraphics[width=9cm]{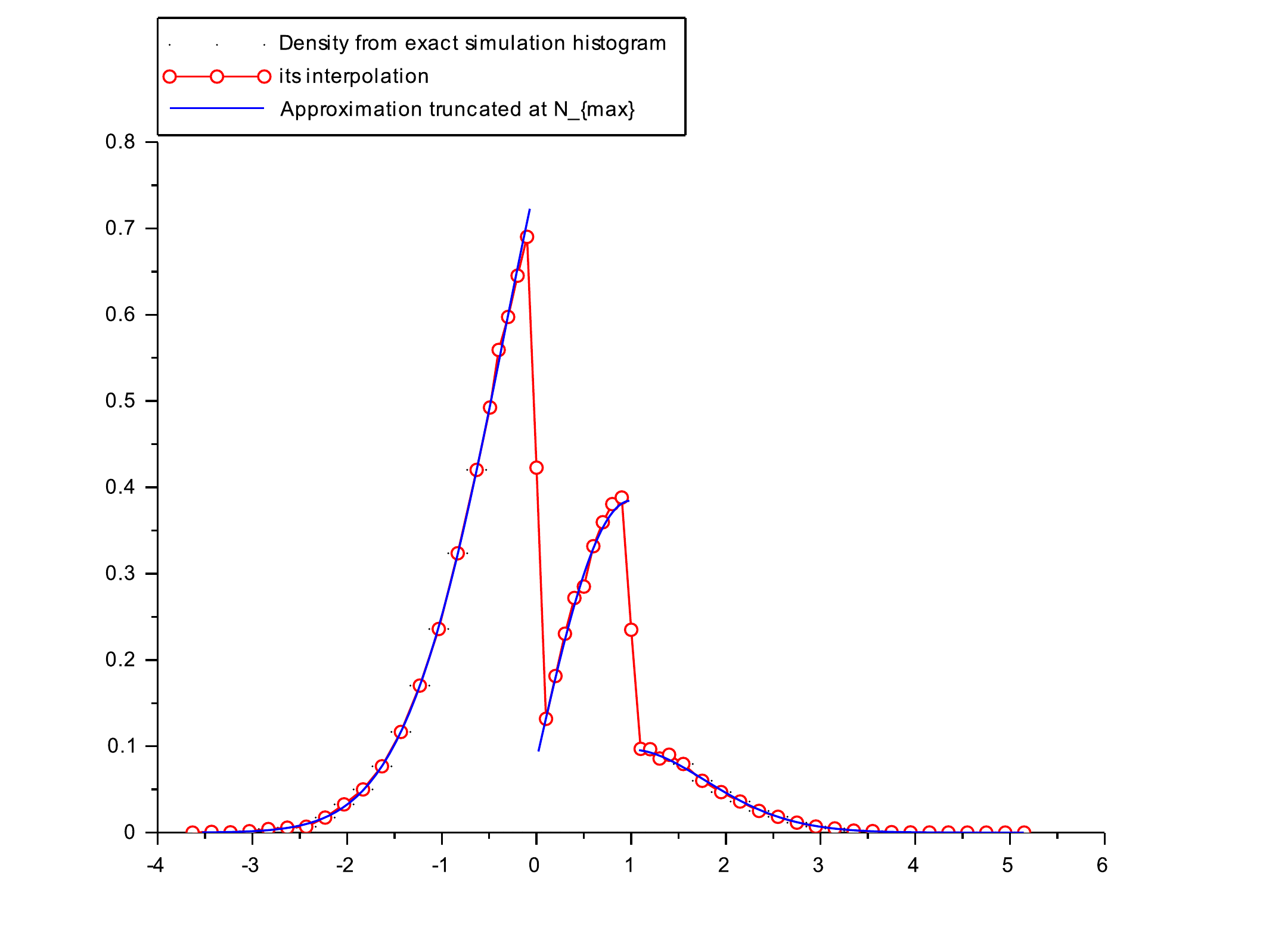}
\caption{Comparison between the function $y \mapsto p^{(-0.8,-0.6)}(1,0.5,y)$ obtained from 50000 exact simulations through generalized rejection sampling method and the plot of its truncated version at $N_{max}=10$ with $\delta_{10} = 3.12\cdot 10^{-4} $ (and $\overline{v}=5.54$). The barriers are $z_1=0$ and $z_2=1$. The average acceptance number is $3.58.$}\label{concordparam}
\end{minipage}
\hfill
\begin{minipage}{0.45\textwidth}
\includegraphics[width=9cm]{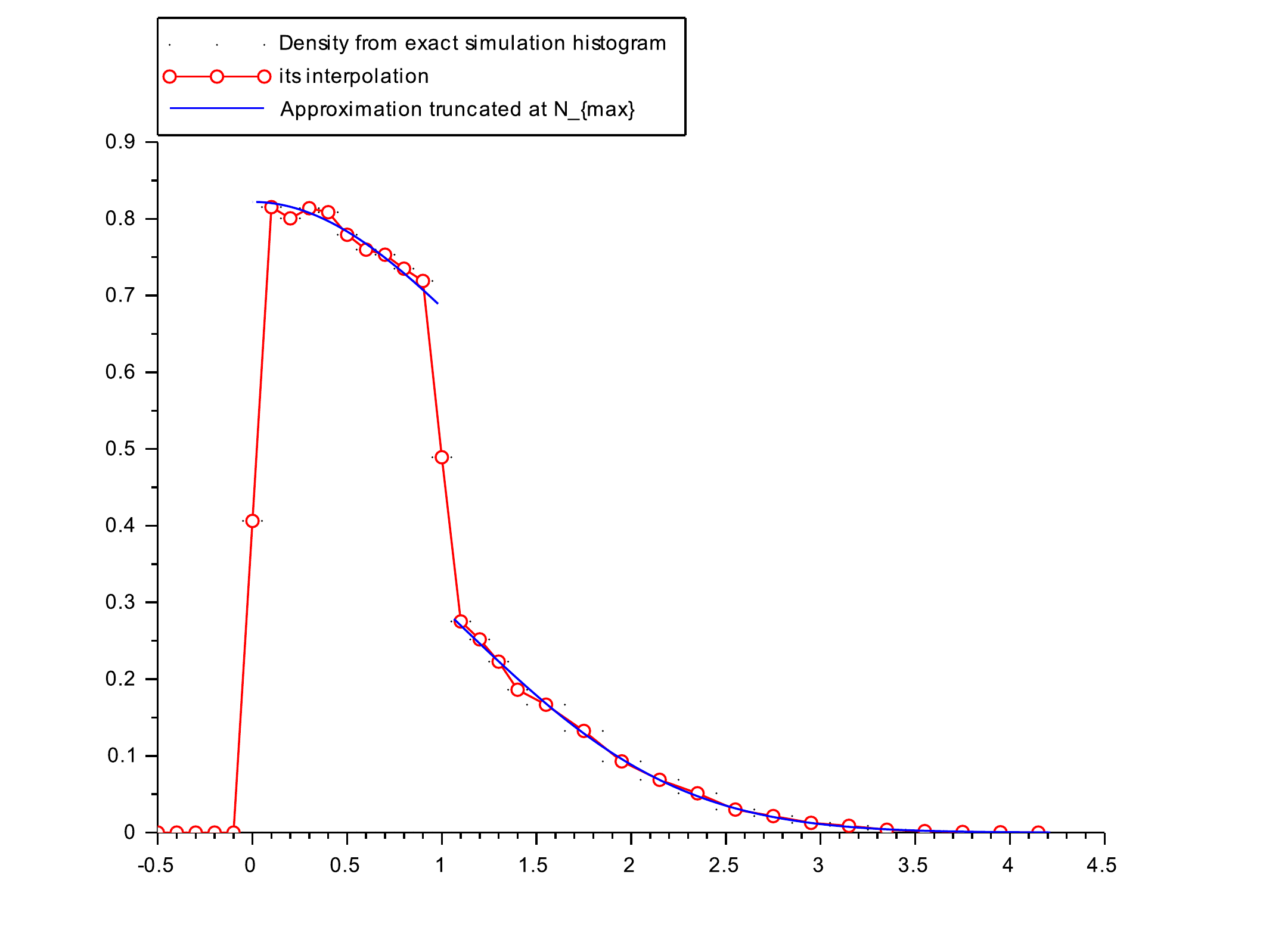}
\caption{Comparison between the function $y \mapsto p^{(1,-0.4)}(1,0.5,y)$ obtained from 50000 exact simulations through generalized rejection sampling method and the plot of its truncated version at $N_{max}=20$. The barriers are $z_1=0$ (completely reflecting) and $z_2=1$ (semipermeable). The average acceptance number is $2.36.$}\label{permeablereflection}
\end{minipage}
\end{figure}

Figure~\ref{concordparam} represents a case of $\beta_1\beta_2>0$. From the simulated density function it is confirmed the behaviour we expected: the process after a time $t$ will be more likely to stay on the the left (respectively right if the parameters are positive) side of the barriers. We chose the parameters $\beta_1<\beta_2$ in such a way that the process would more likely stay in $(-\infty,z_1)$.

Another interesting example is the case of a completely reflecting barrier and a partially reflecting one: in Figure~\ref{permeablereflection} we choose $\beta_1=1$ and $\beta_2<0$, i.e. $z_1$ totally reflecting and $z_2$ semipermeable with semipermeabiliy coefficient $\beta_2=-0.4$. The process shows the tendency to stay in the between the barriers $(z_1,z_2)$, while it will have probability zero to be in $(-\infty,z_1)$.

\section*{Appendix: details in the proof of Theorem \ref{tdf2skewdrift}}

We now propose with more details the steps between the convolution of Fourier transforms (\ref{firststepfourier}) and the final result of Theorem \ref{tdf2skewdrift}.

Equation (\ref{secondstepfourier}) is the convolution of two Fourier transforms, hence one needs to compute first the Fourier transforms separately and then the convolution.

\begin{lemma} The Fourier transform $\mathcal{F}$ of the function $w \mapsto e^{-\frac{w^2}{2}} c_j(y,\mu\sqrt{t};i w) (w^2+\mu^2t)^k$ is
\begin{equation} \mathcal{F}(v)=\sum_{h=0}^2 c_{j,2-h}(y) \sum_{m=0}^k {k \choose m} (\mu\sqrt{t})^{2(k-m+1)-h}  (-1)^{m+h} \frac{d^{2 m +h}}{d v^{2 m + h}} e^{-\frac{v^2}{2}},\end{equation}
where the functions $c_{j}$ and $c_{j,h}$ for $j=1,2,3,4$, $h=0,1,2$ are given in Lemma \ref{Green2skew}.

\end{lemma}
\begin{proof}
Simply recall that
\[ c_j(y,\mu\sqrt{t};i w) = \sum_{h=0}^2 (\mu\sqrt{t})^{2-h} c_{j,2-h}(y) i^h w^h, \] and that
\[ \begin{split} i^h\mathcal{F}\left(  e^{-\frac{w^2}{2}} w^h (w^2+\mu^2t)^k\right) & = \sum_{m=0}^k {k \choose m} (\mu\sqrt{t})^{2(k-m)}  i^h \mathcal{F}\left(  e^{-\frac{w^2}{2}}  w^{2 m+h}\right). \end{split}\]
Finally one computes the Fourier transforms
\[ \begin{split} i^h \mathcal{F}\left(  e^{-\frac{w^2}{2}}  w^{2 m+h}\right) (v) & =  i^{2 (m +h)} \frac{d^{2 m +h}}{d v^{2 m + h}} e^{-\frac{v^2}{2}} =(-1)^{m+h} \frac{d^{2 m +h}}{d v^{2 m + h}} e^{-\frac{v^2}{2}} \end{split}\]
and concludes.
\end{proof}

If $\beta_1\neq \beta_2$, as corollary of Lemma \ref{PFD} and Lemma \ref{Fourier2} one has
\begin{equation} 
\begin{split}
\mathcal{F} & \left(w\mapsto\frac{-1}{(w- i \beta_1 \mu \sqrt{t})^{k+1}(w-i\beta_2 \mu \sqrt{t})^{k+1}}\right)(\omega)=\\ 
& = \frac{\sqrt{2 \pi}}{(\beta_1-\beta_2)^{2 k +1} (\mu \sqrt{t})^{2k+1} k!} \cdot \sum_{n=0}^{k} \frac{\left(2k-n\right)!}{n! (k-n)!} (\beta_1-\beta_2)^n (\mu \sqrt{t})^n \omega^{n} \left[g(\omega,{\beta_2 \mu}\sqrt{t}) - (-1)^n g(\omega,{\beta_1 \mu}\sqrt{t})\right],
\end{split}
\end{equation}
where we defined the function
\[ g(\omega,A):=  \left(2\mathbbm{1}_{\mathbb{R}^+}(A)-1\right)e^{A \, \omega} \mathbb{I}_{\mathbb{R}^-}\left(\left(2\mathbbm{1}_{\mathbb{R}^+}(A)-1\right)\omega\right)\]
Since $\beta_1\mu, \beta_2\mu$ are both positive then $g(\omega,\beta_i \mu\sqrt{t})= e^{\beta_i\mu\sqrt{t}\omega} \mathbb{I}_{\mathbb{R}^-}(\omega)$, but we give here the proof in the general case $\beta_i\mu\sqrt{t}\neq 0$.

If $\beta_1=\beta_2$ Lemma \ref{Fourier2} gives the formula for $\mathcal{F}\left(\frac{-1}{(w- i \beta_1\mu\sqrt{t})^{2(k+1)}}\right)$.

Let us define
\[\mathscr{G}_{m,n}^h(\omega,A):=(-1)^h\left(w^{n}g(w,A)*  \frac{d^{2 m +h}}{d w^{2 m + h}} e^{-\frac{w^2}{2}}\right)(\omega)\]
where $A$ is a fixed real parameter and also \[\mathscr{F}_{m,n}^h(\omega)=\mathscr{G}_{m,n}^h(\omega,\beta_2\mu \sqrt{t})-(-1)^n \mathscr{G}_{m,n}^h(\omega,\beta_1\mu \sqrt{t}).\]
$F_{j,k}$ in equation (\ref{secondstepfourier}) is given by
\[\begin{cases}
F_{j,k}= \sum_{n=0}^k\sum_{m=0}^k \frac{  (-1)^m (2k-n)!}{(k-n)! (k-m)! n! m! k! }\frac{(\mu\sqrt{t})^{2(k-m)}}{(\beta_1\mu\sqrt{t}-\beta_2\mu\sqrt{t})^{2k+1-n}} \sum_{h=0}^2 c^{2-h}_j(x,y)(\mu\sqrt{t})^{2-h}\mathscr{F}_{m,n}^h, & \text{ if } \beta_1\neq \beta_2,\\
F_{j,k}= \sum_{m=0}^k \frac{  (-1)^{m+k}}{(2k+1)! }{k\choose m} (\mu\sqrt{t})^{2(k-m)} \sum_{h=0}^2 c^{2-h}_j(x,y)(\mu\sqrt{t})^{2-h}\mathscr{G}_{m,2k+1}^h, & \text{ if }\beta_1= \beta_2.
\end{cases} \]

It remains to compute the function $\mathscr{G}_{m,n}^h(\omega,A)$.
One can use that
\[\frac{d^{n}}{d w^{n}} e^{-\frac{w^2}{2}}=(-1)^n e^{-\frac{w^2}{2}} H_{n}(w)\]
where $H_n(w)$ are the Hermite polynomials.
\[\begin{split}
\mathscr{G}_{m,n}^h(\omega,A) &=\left(w^{n}g(w,A)* H_{2m+h}(w)e^{-\frac{w^2}{2}}\right)(\omega)\\
&=\left(2\mathbbm{1}_{\mathbb{R}^+}(A)-1\right)\int_{\mathbb{R}} \mathbb{I}_{\mathbb{R}^-}(A w) w^{n} e^{-\frac{(\omega-w)^2}{2}+Aw}  H_{2m+h}(\omega-w)dw\\
& =\left(2\mathbbm{1}_{\mathbb{R}^+}(A)-1\right)e^{\frac{A^2}{2}+A\omega} \int_{\mathbb{R}} \mathbb{I}_{\mathbb{R}^-}(A w) w^{n} e^{-\frac{\left(w-(\omega+A)\right)^2}{2}}  H_{2m+h}(\omega-w)dw\\
& \overset{(v=w-A-\omega)}{=} \left(2\mathbbm{1}_{\mathbb{R}^+}(A)-1\right)e^{\frac{A^2}{2}+A\omega} \int_{\mathbb{R}} \mathbb{I}_{\mathbb{R}^-}(A (v+A+\omega)) (v+A+\omega)^{n} e^{-\frac{v^2}{2}}  H_{2m+h}(-A-v)dv\\
&= \left(2\mathbbm{1}_{\mathbb{R}^+}(A)-1\right)e^{\frac{A^2}{2}+A\omega} \int_{(-\left(2\mathbbm{1}_{\mathbb{R}^+}-1\right)\infty,-(\omega+A))} (v+A+\omega)^{n} e^{-\frac{v^2}{2}}  H_{2m+h}(-A-v)dv
\end{split}\]
We can then use the binomial formula and then the explicit expression for the Hermite polynomials:
\[H_n(w)=n! \sum_{\ell=0}^{\lfloor {\frac{n}{2}}\rfloor} (-1)^\ell \frac{1}{2^\ell}\frac{1}{\ell! (n-2\ell)!} w^{n-2\ell}.\]

Therefore
\[\begin{split}
\mathscr{G}_{m,n}^h(\omega,A) &= \left(2\mathbbm{1}_{\mathbb{R}^+}(A)-1\right) e^{\frac{A^2}{2}+A\omega} \int_{(-\left(2\mathbbm{1}_{\mathbb{R}^+}-1\right)\infty,-(\omega+A))} (v+A+\omega)^{n} e^{-\frac{v^2}{2}}  H_{2m+h}(-(A+v))dv\\
 &= (2m+h)! \sum_{\ell=0}^{m+\lfloor {\frac{h}{2}}\rfloor} \frac{(-1)^\ell }{2^\ell}\frac{1}{\ell! (2(m-\ell)+h)!} \ S_{m,n,l}^h(A,\omega)
\end{split}\]
where, if $I(\omega,A)$ denotes the interval $(-\left(2\mathbbm{1}_{\mathbb{R}^+}(A)-1\right)\infty,-(\omega+A))$,
\[\begin{split}
 S_{m,n,l}^h(A,\omega) & =(-1)^h \left(2\mathbbm{1}_{\mathbb{R}^+}(A)-1\right) e^{\frac{A^2}{2}+A\omega} \int_{I(\omega,A)} (v+A+\omega)^{n} e^{-\frac{v^2}{2}} (A+v)^{2(m-\ell)+h}dv\\
& = (-1)^h \sum_{r=0}^n \sum_{s=0}^{2(m-\ell)+h} {n \choose r} {2(m-\ell)+h \choose s} (\omega+A)^{n-r} A^{2(m-\ell)+h-s} J_{r+s}(\omega,A),
\end{split}\]
where $J_{r+s}(\omega,A)=\left(2\mathbbm{1}_{\mathbb{R}^+}(A)-1\right)e^{\frac{A^2}{2}+A\omega}\int_{I(\omega,A)} v^{r+s} e^{-\frac{v^2}{2}}dv=\begin{cases}
e^{\frac{A^2}2+A\omega}I_{r+s}(-(\omega+A)) & A>0\\
- e^{\frac{A^2}2+A\omega}\tilde{I}_{r+s}(-(\omega+A)) & A<0\end{cases}$ with $I_q$ and $\tilde I _q$ defined in Lemma \ref{integraldef}.\\

\textbf{Acknowledgements}: The authors acknowledge the Deutsch-Französische Hochschule - Université Franco-Allemande (DFH-UFA) and the RTG 1845 Stochastic Analysis with Applications in Biology, Finance and Physics for their financial support.

\addcontentsline{toc}{section}{Bibliography}

\bibliographystyle{plain}
\bibliography{exact_algorithm_bibliography}

\end{document}